\documentclass[a4paper,12pt]{article}
\usepackage[margin=3cm,footskip=1cm]{geometry}

\usepackage{lmodern}
\usepackage{dsfont}
 \usepackage[utf8]{inputenc} 
\usepackage[T1]{fontenc}
\usepackage{microtype}
\usepackage{syntonly}
\usepackage{float} 
\usepackage{xcolor}
\usepackage{amsmath,amssymb,amsfonts,amsthm,bm,mathtools,xspace,booktabs}
\usepackage{graphicx,subfig,color}
\usepackage{pdfcomment}
\usepackage{enumitem}
\usepackage[normalem]{ulem}
\usepackage[mathlines,pagewise]{lineno}

\usepackage{authblk}

\let\oldabstract\abstract
\makeatletter
\renewcommand\abstract{%
  \providecommand\keywords{\par\medskip\noindent\textit{Keywords:}\xspace}
  \oldabstract\noindent\ignorespaces}
\makeatother

\numberwithin{equation}{section}

\theoremstyle{plain} 
\newtheorem{theorem}{Theorem}[section]
\newtheorem{lemma}[theorem]{Lemma}
\newtheorem{proposition}[theorem]{Proposition}
\newtheorem{definition}[theorem]{Definition}
\newtheorem{corollary}[theorem]{Corollary}

\usepackage{algorithmic, algorithm}

\def\R{\mathbb R}
\def\E{\mathbb E}

\def\F{\mathcal F}
\def\K{\mathcal K}
\def\N{\mathbb N}

\def\1{\mathbf{1}}

\def\X{\mathbf{X}}

\def\Kn{\widehat{K}_n}
\def\gn{\widehat{g}_n}
\def\Jn{\widehat{J}_n}
\def\wtheta{\widehat{\theta}_n}
\def\ttheta{\widetilde{\theta}_n}
\def\wrho{{\widehat{\rho}_n}}

\DeclareMathOperator*{\argmin}{arg\,min}  
\def\rm{{r_{min}}}
\def\rM{{r_{max}}}
\def\Dn{{|D_n|}}
\def\convl{\xrightarrow[n\rightarrow +\infty]{distr.}}
\def\convP{\xrightarrow[n\rightarrow + \infty]{\mathbb{P}}}
\def\convPs{\xrightarrow[n\rightarrow +\infty ]{a.s.}}
\def\cum{\mathrm{Cum}}
\graphicspath{{.}{./figure_quantifying_rep_DPP_supplementary/}{./figure_quantifying_rep_DPP/}{./figure_intro/}{./Figures/}}

\begin{document}

\title{Contrast estimation for parametric stationary determinantal point processes}

\author[1]{Christophe Ange Napol\'eon Biscio}
\author[1,2]{Fr\'ed\'eric Lavancier}
\affil[1]{ Laboratoire de Math\'ematiques Jean Leray\\
University of Nantes, France}
 \affil[2]{ Inria, Centre Rennes Bretagne Atlantique, France}

\date{}

\maketitle

\begin{abstract}
We study minimum contrast estimation for parametric stationary determinantal point processes. These processes form a useful class of models for repulsive (or regular, or inhibitive) point patterns and are already applied in numerous statistical applications.  Our main focus is on minimum contrast methods based on the Ripley's $K$-function or on the pair correlation function. Strong consistency and asymptotic normality of theses procedures are proved under general conditions that only concern the existence of the process and its regularity with respect to the parameters. A key ingredient of the proofs is the recently established Brillinger mixing property of stationary determinantal point processes. This work may be viewed as a complement to the study of Y. Guan and M. Sherman who establish the same kind of asymptotic properties for  a large class  of Cox processes, which in turn are models for clustering (or aggregation). 

   \end{abstract}

\section{Introduction}

Determinantal point processes (DPPs) are models for repulsive (or regular, or inhibitive) point processes data.
They have been introduced by O. Macchi in~\cite{macchi1975coincidence} to model the position of fermions, which are particles that repel each others. Their probabilistic aspects have been studied thoroughly, in particular in~\cite{Soshnikov:00}, \cite{ShiraiTakahashi1:03} and~\cite{hough2009zeros}. 
Recently, DPPs have been studied and applied from a statistical perspective.  
A description of their main statistical aspects is conducted in~\cite{lavancierpublish} and they actually  turn out to be a well-adapted statistical model in domains  as  statistical learning  \cite{Kulesza:Taskar:12}, telecommunications \cite{deng2014ginibre,Miyoshi:Shirai13}, biology and ecology (see the examples in~\cite{lavancierpublish} and \cite{lavancier_extended}).

A DPP is defined through a kernel $C$, basically a covariance function. Assuming a parametric form for $C$, several estimation procedures are considered in \cite{lavancierpublish}, specifically the maximum likelihood method and minimum contrast procedures based on  the Ripley's $K$ function or the pair correlation $g$.  
These methods are implemented in the \texttt{spatstat} library \cite{BaddelyRubakTurner15,baddeley:turner:05}
of  R \cite{R:15}.
From  the simulation study conducted in \cite{lavancierpublish} and \cite{lavancier_extended},  see also Section~\ref{estimation}, the maximum likelihood procedure seems to be the best method  in terms of quadratic loss.  
However, the expression of the likelihood relies in theory on a spectral representation of $C$, which is rarely known in practice, and some Fourier approximations are introduced in  \cite{lavancierpublish}. The likelihood  also involves the determinant of a $n\times n$ matrix, where $n$ is the number of observed points, which is prohibitively time consuming to compute when $n$ is large. In contrast, the estimation procedures based on   $K$ or $g$ do not require the knowledge of any spectral representation of $C$ and are faster to compute in presence of large datasets, which explain their importance in practice.

From a theoretical point of view, neither the likelihood method nor the minimum contrast methods for DPPs have been studied thoroughly, even in assuming that a spectral method for $C$ is known. In this work, we focus on parametric stationary DPPs and we prove the strong consistency and the asymptotic normality of the minimum contrast estimators based on $K$ and $g$. These questions are in connection with the general investigation of  Y. Guan and M. Sherman~\cite{GuanSherman:07}, who study the asymptotic properties of the latter estimators for stationary point processes. However the setting in~\cite{GuanSherman:07} has a clear view to Cox processes and the assumptions involve both $\alpha$-mixing and Brillinger mixing conditions, which are indeed satisfied for a large class of Cox processes. Unfortunately these results do no apply straightforwardly to DPPs. We consider instead more general versions of the asymptotic theorems in~\cite{GuanSherman:07} and we prove that they apply nicely to DPPs. Our main ingredient then becomes the Brillinger mixing property of stationary DPPs, recently proved in~\cite{bisciobrillingerTCL}, and we do not need  any $\alpha$-mixing condition. 
Our asymptotic results finally gather a very large class of stationary DPPs, where the main assumptions are quite standard and only concern the regularity of the kernel $C$ with respect to the parameters. 
As an extension to the results in~\cite{GuanSherman:07}, it is worth mentioning the study of~\cite{Waagepetersentwostepestimation}  dealing with constrast estimation for some  inhomogeneous spatial point processes, still under a crucial $\alpha$-mixing condition. We do not address this generalization for DPPs in the present work. 

The remainder of this paper is organized as follows. In Section~\ref{dpps_estimation}, we recall the definition of stationary DPPs, some of their basic properties and we discuss parametric estimation of DPPs. Our main results are presented in Section~\ref{section special case minimum contrast}, namely the asymptotic properties of the minimum contrast estimators of a DPP based on the $K$ or the $g$ function. Section~\ref{proofs} gathers the proofs of our main results. In the appendix, we finally  present our general asymptotic result for minimum contrast estimators and some auxiliary materials.

\section{Stationary DPPs and parametric estimation}\label{dpps_estimation}

\subsection{Stationary DPPs}\label{dpps}
 We refer to~\cite{daleyvol1,daleyvol2} for a general presentation on point processes. Let $\X$ be a simple point process on $\R^d$. For a bounded set $D\subset\R^d$, denote by $\X(D)$ the number of points of $\X$ in $D$ and let $E$ be the expectation over the distribution of $\X$.
 If there exists a function $\rho^{(k)}: (\R^d)^k \to \R^+$, for $k\geq 1$,  such that for any family of mutually disjoint subsets $D_1,\dots,D_k$ in $\R^d$
\[E \prod_{i=1}^k \X(D_i) = \int_{D_1}\dots\int_{D_k} \rho^{(k)}(x_1,\dots,x_k) dx_1\dots dx_k,\]
then this function is called the joint intensity of order $k$   of $\X$. 
If $\X$ is stationary, $ \rho^{(k)}(x_1,\dots,x_k)= \rho^{(k)}(0,x_2-x_1,\dots,x_k-x_1)$ and in particular $\rho^{(1)}=\rho$ is a constant. 
From its definition, the  joint intensity of order $k$ is unique up to a Lebesgue nullset.  Henceforth, for ease of
presentation, we ignore 
nullsets. In particular  we will say that a function is continuous whenever
there exists a continuous version of it.

Determinantal point processes (DPPs) are defined through their joint intensities. We refer to the survey by Hough et al. \cite{hough2009zeros} for a general presentation including the non-stationary case and the extension to complex-valued kernels. 
 We focus in this work on stationary DPPs and so we restrict the definition to this subclass. We also consider for simplicity real-valued kernels.  

\begin{definition}\label{DPPdefinition intro}
Let $C: \R^d \rightarrow \R$ be a function.
A point process $\X$ on $\R^d$ is  a stationary DPP with kernel $C$, in short $\X\sim DPP(C)$,  if for all $k\geq 1$ its joint intensity of order $k$ exists and satisfies the relation
\begin{align*}
 \rho^{(k)}(x_1,\ldots x_k)=\det [C](x_1,\dots,x_k)
\end{align*}
for every $(x_1,\dots,x_k)\in \R^{dk}$, where $[C](x_1,\dots,x_k)$ denotes the matrix with entries $C(x_i-x_j)$,  $1\leq i,j\leq k$.
\end{definition}

Conditions on $C$ ensuring the existence of $DPP(C)$ are recalled in the next proposition.  We define the Fourier transform of a function $h\in L^1(\R^d)$ as 
 \begin{align*}
  \F(h)(t)=\int_{\R^d} h(x) e^{2i\pi x\cdot t}dx,\quad \forall t \in \R^d
 \end{align*}
and we consider its extension to  $L^2(\R^d)$ by Plancherel's theorem, see~\cite{stein1971fourier}.

\medskip

\noindent{\bf Condition $\K(\rho)$.} A kernel $C$ is said to verify condition $\K(\rho)$ if $C$ is a symmetric continuous real-valued function in $ L^2(\R^d)$ with $C(0)=\rho$ and $0\leq \F(C)\leq 1$.


\begin{proposition}[\cite{Soshnikov:00, lavancierpublish}]\label{DPPexistence intro}
Assume $C$ satisfies $\K(\rho)$. Then $DPP(C)$ exists and is unique if and only if  $0\leq \F(C)\leq 1$.
\end{proposition}

In short, $DPP(C)$ exists whenever $C$ is a continuous covariance function in $L^2(\R^d)$ with $\F(C)\leq 1$. This  makes easy the construction of parametric families of DPPs, simply  considering parametric families of covariance functions where the condition $\F(C)\leq 1$ appears as a constraint on the parameters. Some examples are given in \cite{lavancierpublish}, \cite{bisciobernoulli} and in the next section. 

By definition, all moments of a DPP are known, in particular the pair correlation (pcf) and the Ripley's $K$-function can explicitly be expressed in terms of the kernel.  For $C$ satisfying $\K(\rho)$, let $R(x)=C(x)/C(0)$ be the correlation function associated to $C$.
The pcf, defined in the stationary case for all $x\in\R^d$ by $g(x) = {\rho^{(2)}(0,x)}/{\rho^2}$, writes
\[ g(x) = 1 - R^2(x).\]
The Ripley's $K$-function is in turn given for all $t\geq 0$ by
 \begin{align}\label{defK}
  K(t) = \int_{B(0,t)} g(x) dx = \int_{B(0,t)} (1 - R^2(x)) dx
 \end{align}
  where $B(0,t)$ is the Euclidean ball centred at $0$ with radius $t$. 
For later purposes, we denote by  $c_{[k]}^{red}$  the density of the reduced factorial cumulant moment measures  of order $k$ of $\X$. 
We  refer to \cite{bisciobrillingerTCL} for the definition and further details, where the following particular cases are derived. Assuming that the kernel $C$ of $\X$ satisfies $\K(\rho)$, we have for all $(u,v,w) \in \R^{3d}$
\begin{align}
& c^{red}_{[2]}(u) = - C^2(u), \label{expression densite cumulant 2}\\
& c^{red}_{[3]} (u,v) = \ 2 \ C(u)C(v)C(v-u), \label{expression densite cumulant 3} \\
& c^{red}_{[4]} (u,v,w) = -2 \big[  C(u)C(v)C(u-w)C(v-w) + C(u) C(w) C(u-v) C(v-w) \notag   
   \\&\hspace{2cm}+ C(v) C(w) C(u-v) C(u-w)  \big]. \label{expression densite cumulant 4}
\end{align}

\subsection{Parametric estimation of DPPs}\label{estimation}

We consider a parametric family of DPPs with kernel $C_{\rho,\theta}$ where $\rho=C_{\rho,\theta}(0)>0$ and  $\theta$ belongs to a subset $\Theta_{\rho}$ of $\R^p$, for a given $p\geq1$. To ensure the existence of $DPP(C_{\rho,\theta})$, we assume that for all $\rho>0$ and any $\theta \in \Theta_\rho$, the kernel $C_{\rho,\theta}$ verifies $\K(\rho)$, which explains the indexation of $\Theta_{\rho}$ by $\rho$. We assume further that for a given $\rho_0>0$ and $\theta_0$ in the interior of $\Theta_{\rho_0}$ (provided this interior is non-empty) we observe the point process $\X \sim DPP(C_{\rho_0,\theta_0})$ on a bounded domain $D_n \subset \R^d$. 

\medskip

The standard estimator of the intensity $\rho_0$ is  
\begin{align}\label{estimation rho}
 \wrho= \frac{1}{|D_n|} \sum_{x\in \X} \1_{\lbrace x \in D_n\rbrace}
\end{align}
where $|D_n|$ denotes the Lebesgue volume of $D_n$.
Since a stationary DPP is ergodic, see~\cite{Soshnikov:00},  this estimator is strongly consistent by the ergodic theorem, and  it is asymptotically normal, cf \cite{SoshnikovGaussianLimit} and \cite{bisciobrillingerTCL}.  In the following, we focus  our attention on the estimation of $\theta_0$. 
As explained in~\cite{lavancier_extended}, likelihood inference is in theory feasible if we know a spectral representation of $C_{\rho,\theta}$ on $D_n$. Unfortunately no spectral representations are known in the general case and some Fourier approximations are introduced in~\cite{lavancier_extended}. 
Another option is to consider minimum contrast estimators (MCE) as described below.

\medskip

For $\rho>0$ and $\theta\in \Theta_{\rho}$, let $J(.,\theta)$ be a function from $\R^d$ into $\R^+$ which is a summary statistic of $DPP(C_{\rho,\theta})$ that does not depend on $\rho$. In the DPP's case, the most important and natural examples are the $K$-function and the pcf $g$, that we study in detail in the following. Consider $\Jn$ an estimator of $J$ from the observation of $\X$ on $D_n$. Further, let $c\in \R$, $c\neq 0$, be a parameter such that $\Jn(t)^c$ and $J(t,\theta)^c$ are well defined for all  $t\in \R$ and $\theta\in\Theta_{\rho_0}$. Finally, define for $0\leq \rm < \rM$, the discrepancy measure
\begin{align}\label{expression Un discrepancy measure}
 U_n(\theta) = \int_{\rm}^{\rM} w(t) \left\lbrace \Jn(t)^c - J(t,\theta)^c \right\rbrace^2 dt
\end{align}
where $w$ is a smooth weight function. 
The MCE of $\theta_0$ is  
\begin{align}\label{definition thetan}
 \wtheta = \argmin_{\theta \in \Theta_{\hat \rho _n}}  U_n(\theta).
\end{align}

For example, let us consider  the parametric family of DPPs with Gaussian kernels 
\begin{align}\label{gaussian kernel intro}
 C(x) = \rho e^{-\left|\frac{ x}{\alpha}\right|^2},\quad x\in\R^d,
\end{align}
where $|.|$ denote the Euclidean norm on $\R^d$, $\rho>0$ and $\alpha \leq 1/(\sqrt{\pi}\rho^{1/d})$, the latter constraint on the parameter space being a  consequence of the existence condition $\F(C)\leq 1$ in $K(\rho)$. Some realizations are shown in Figure~\ref{figure overview DPP intro}. For comparison, we have estimated the parameter $\alpha$ of this model with the MCE \eqref{definition thetan} when $J$ corresponds to  $K$ or  $g$, and with the maximum likelihood method (using the Fourier approximation of the spectral representation of $C$ introduced in~\cite{lavancierpublish}). The estimators of $K$ and $g$, in place of $\Jn$ in \eqref{definition thetan}, are standard and recalled in Sections~\ref{section theorem ripley}-\ref{section theoreme pcf}, see also~\cite[Chapter 4]{mollerstatisticalinference}.
For the tuning parameters, we followed the standard  choice $w(t)=1$, $r_{\min}=0.01$, $r_{\max}$ as one quarter of the side length of the window and $c=0.5$ as recommended in \cite{Diggle:03} for repulsive point processes.
This simulation study has been carried out with the functions implemented in the \texttt{spatstat} library. 
Table~\ref{table MSE gauss} reports the mean squared errors  of the three mentioned methods over $500$ realisations of $DPP(C)$ with $\rho=100$ and $\alpha= 0.01,\ 0.03,\ \frac{1}{10\sqrt{\pi}}$, observed on $[0,1]^2$, $[0,2]^2$ and $[0,3]^2$.

\begin{figure}[h]
\begin{center}
\begin{tabular}{ccc} 
 \includegraphics[scale=0.35]{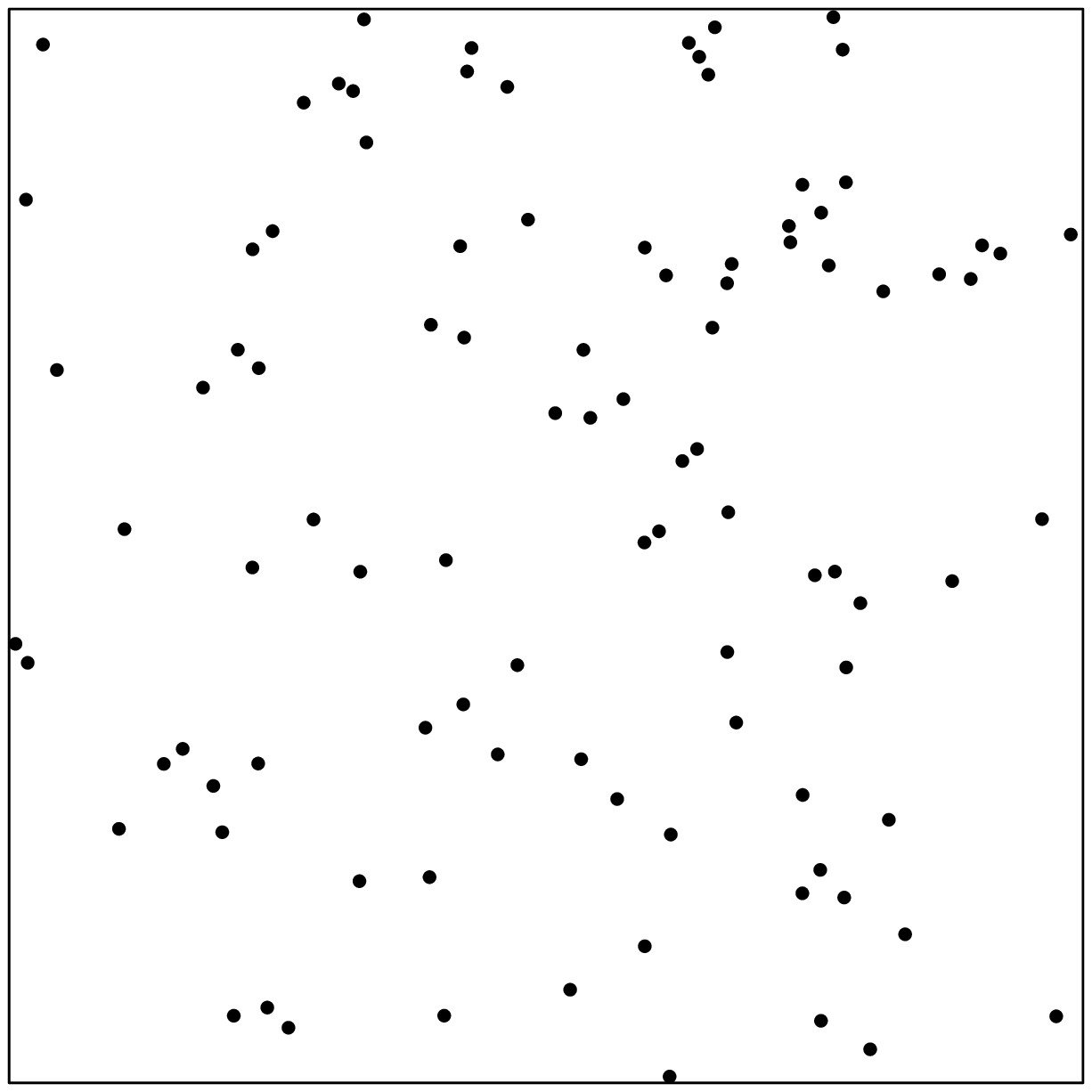} &  \includegraphics[scale=0.35]{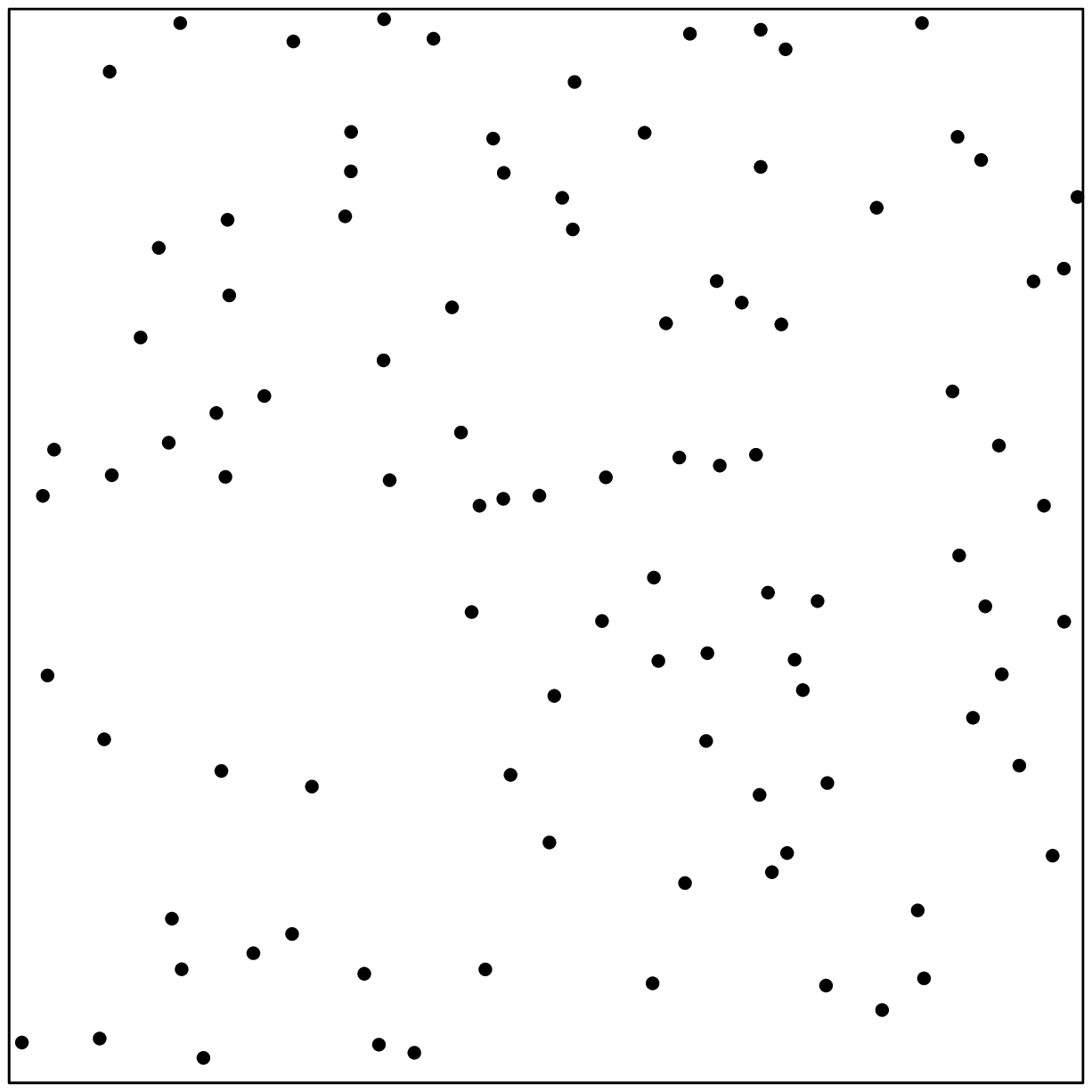}  &  \includegraphics[scale=0.35]{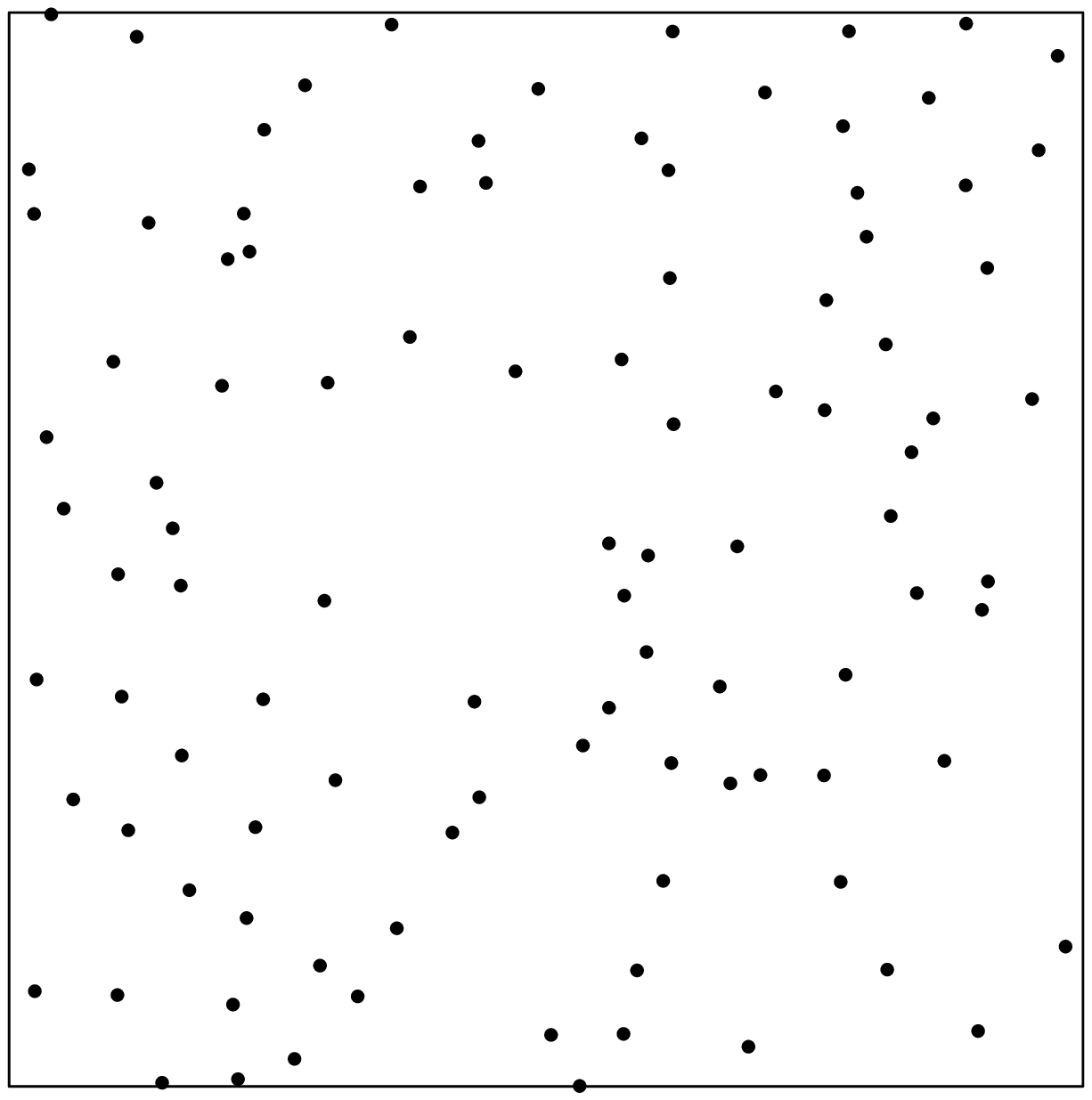}
\end{tabular}
\caption{Realizations on $[0,1]^2$ of DPPs with kernel~\eqref{gaussian kernel intro} where $\rho=100$ and from left to right $\alpha= 0.01,\ 0.03,\ \frac{1}{10\sqrt{\pi}}$.}\label{figure overview DPP intro}
\end{center}
\end{figure}

\begin{table}[h]
\newcommand{\mc}[3]{\multicolumn{#1}{#2}{#3}}
\begin{center}
\def\arraystretch{1.3}
\resizebox{\textwidth}{!} {
\begin{tabular}[c]{l|c c c| c c c| c c c}
\cline{2-10}
 & \mc{3}{c|}{$[0,1]^2$} & \mc{3}{c|}{$[0,2]^2$} & \mc{3}{c|}{$[0,3]^2$}\\\cline{2-10}
\mc{1}{c|}{}                              & $K$ & $g$ & ML & $K$ & $g$ & ML & $K$ & $g$ & \mc{1}{c|}{ML} \\ \cline{1-10}
\mc{1}{|l|}{$\alpha=0.01$}                & 2.026 & 1.039  & 1.032  & 0.848  &  {0.309} & 0.220  &  0.521  & {0.175} & \mc{1}{c|}{0.096}\\
\mc{1}{|l|}{ $\alpha=0.03$}               & 1.214 & 0.706  & 0.786  & 0.419  & {0.248 } & 0.175  & 0.231 & {0.180} & \mc{1}{c|}{0.084}\\
\mc{1}{|l|}{ $\alpha =1/( 10\sqrt{\pi})$} & 0.356  & {0.588} & 0.225  & 0.113  & {0.258 } &  0.061 & 0.051 & {0.176} & \mc{1}{c|}{0.022}\\   \hline
\end{tabular}}
\caption{Mean squared errors of the  MCE \eqref{definition thetan} when $J=K$, $J=g$, and  the maximum likelihood method  estimator (ML) as approximated in~\cite{lavancierpublish}. These values are estimated from 500 realizations of DPPs on $[0,1]^2$,  $[0,2]^2$ and $[0,3]^3$ with kernel~\eqref{gaussian kernel intro}, $\rho=100$ and $\alpha= 0.01,\ 0.03,\ \frac{1}{10\sqrt{\pi}}$. All entries are multiplied by $10^{4}$ to make the table more compact.} \label{table MSE gauss}
\end{center}
\end{table}

For all methods considered in Table~\ref{table MSE gauss}, the estimators seem consistent and the precision, in the sense of the mean squared errors, increases with the size of the observation window. From these results, the maximum likelihood method seems to be the best method in terms of quadradic loss, which agrees with the observations made in~\cite{lavancierpublish}. However, MCEs, especially the one based on $g$, seem to perform reasonably well. Moreover, their computation is faster than the maximum likelihood method and do not rely on an approximated spectral representation of $C$. For instance, with a regular laptop, the estimation of $\alpha$ for 500 realizations on  $[0,3]^2$ took about 30 minutes for the MCEs based on $K$ and $g$ against more than 7  hours by the maximum likelihood method. Finally, it seems that each estimator has an asymptotic Gaussian behaviour, as illustrated in Figure~\ref{figure histogram DPP} where we have represented the histograms obtained from the estimations of $\alpha=0.03$  over 500 realizations on $[0,1]^2$ as in Table~\ref{table MSE gauss}. The remainder of this paper is dedicated to proving the asymptotic normality of the MCE \eqref{definition thetan}  when $J=K$ or $J=g$ and $\X$ is a stationary DPP. The asymptotic properties of the maximum likelihood estimator remain an open problem. Note finally that a solution to improve the efficiency of the MCEs, still avoiding the computation of the likelihood, is to construct an optimal linear combination of the MCE  based on $K$ and the MCE based on  $g$, see \cite{Lavancier:Rochet:2014} for a general presentation of the procedure and \cite{lavancier_moller15} for an example in spatial statistics.

\begin{figure}[H]
\begin{center}
\begin{tabular}{ccc} 
 \includegraphics[scale=0.34]{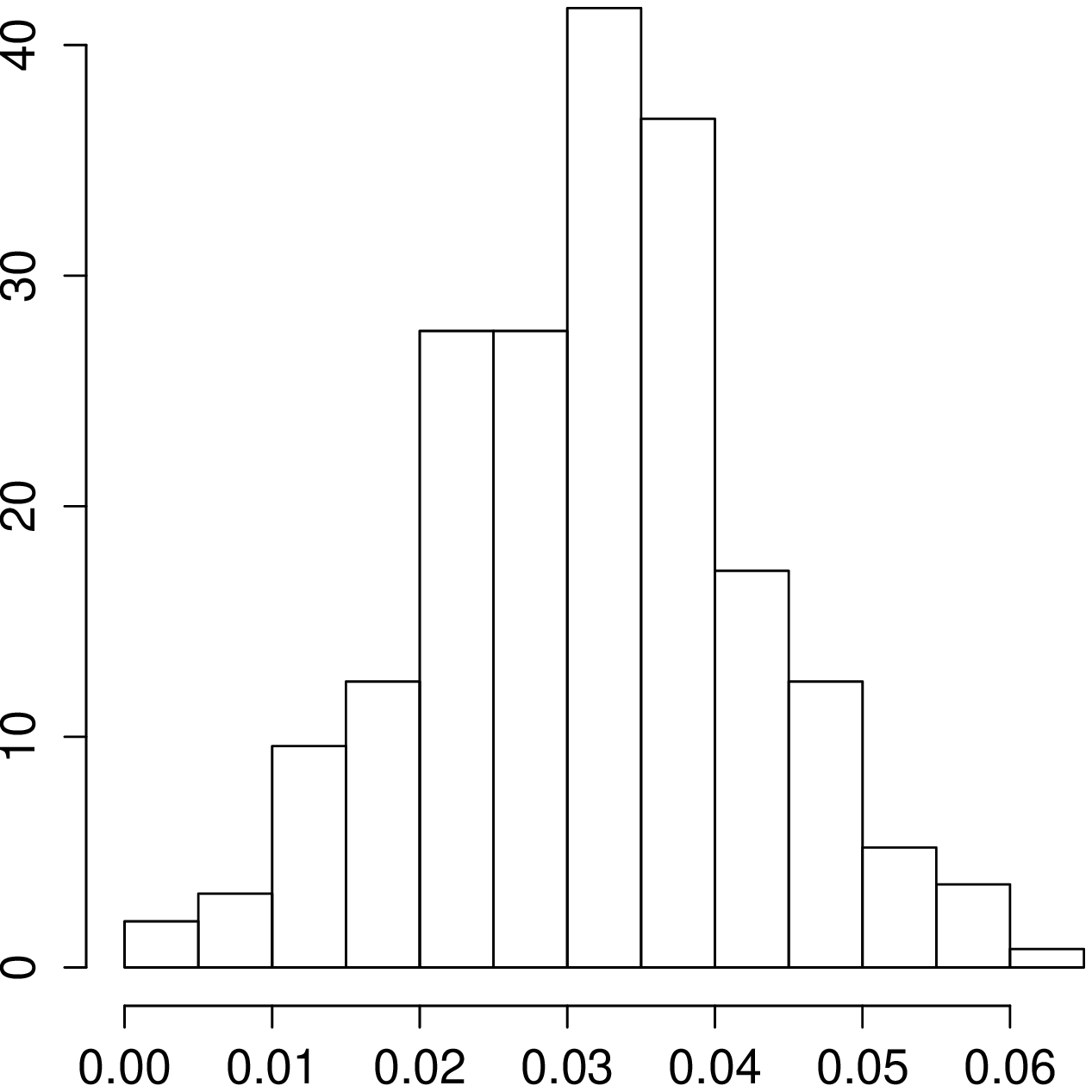} &  \includegraphics[scale=0.34]{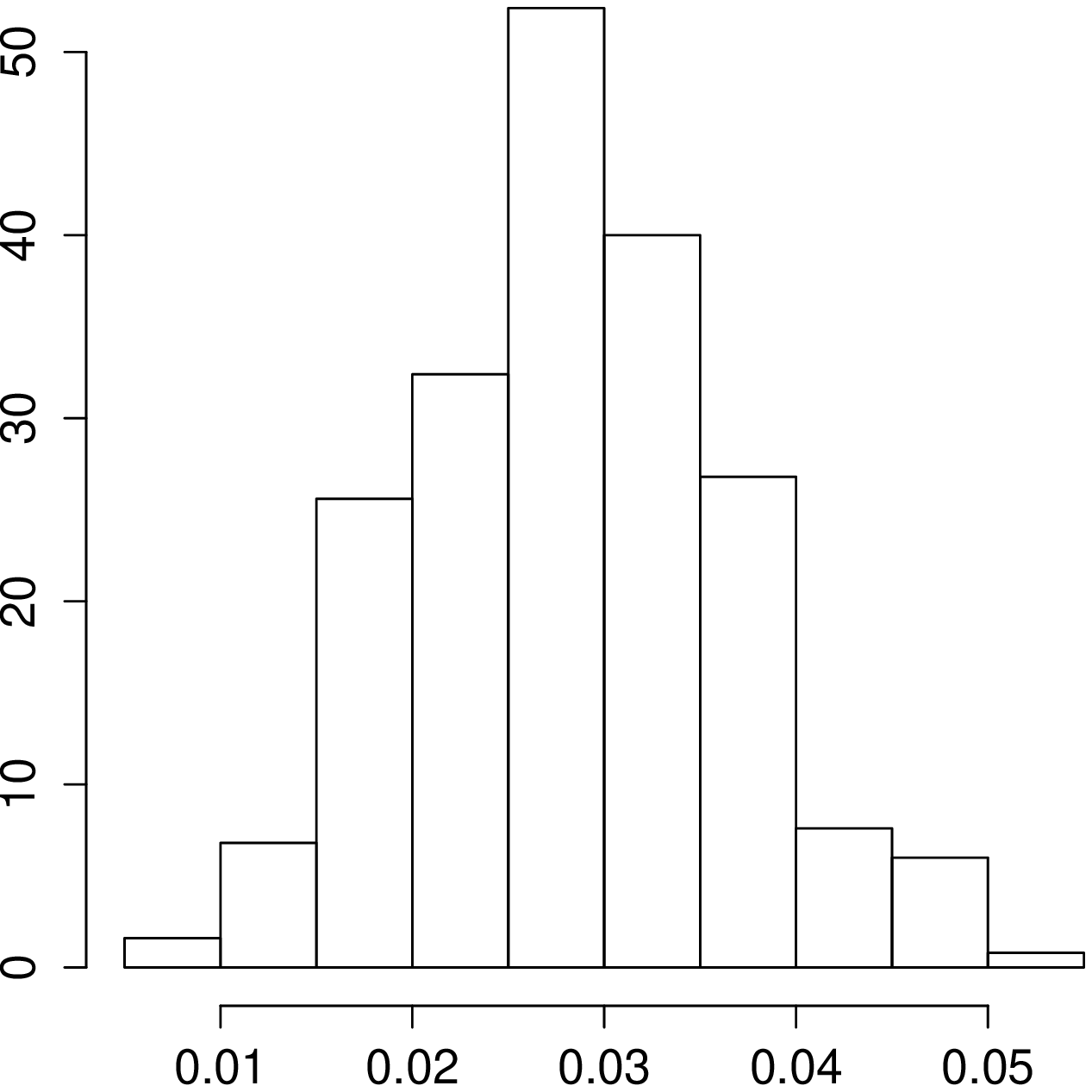}  &  \includegraphics[scale=0.34]{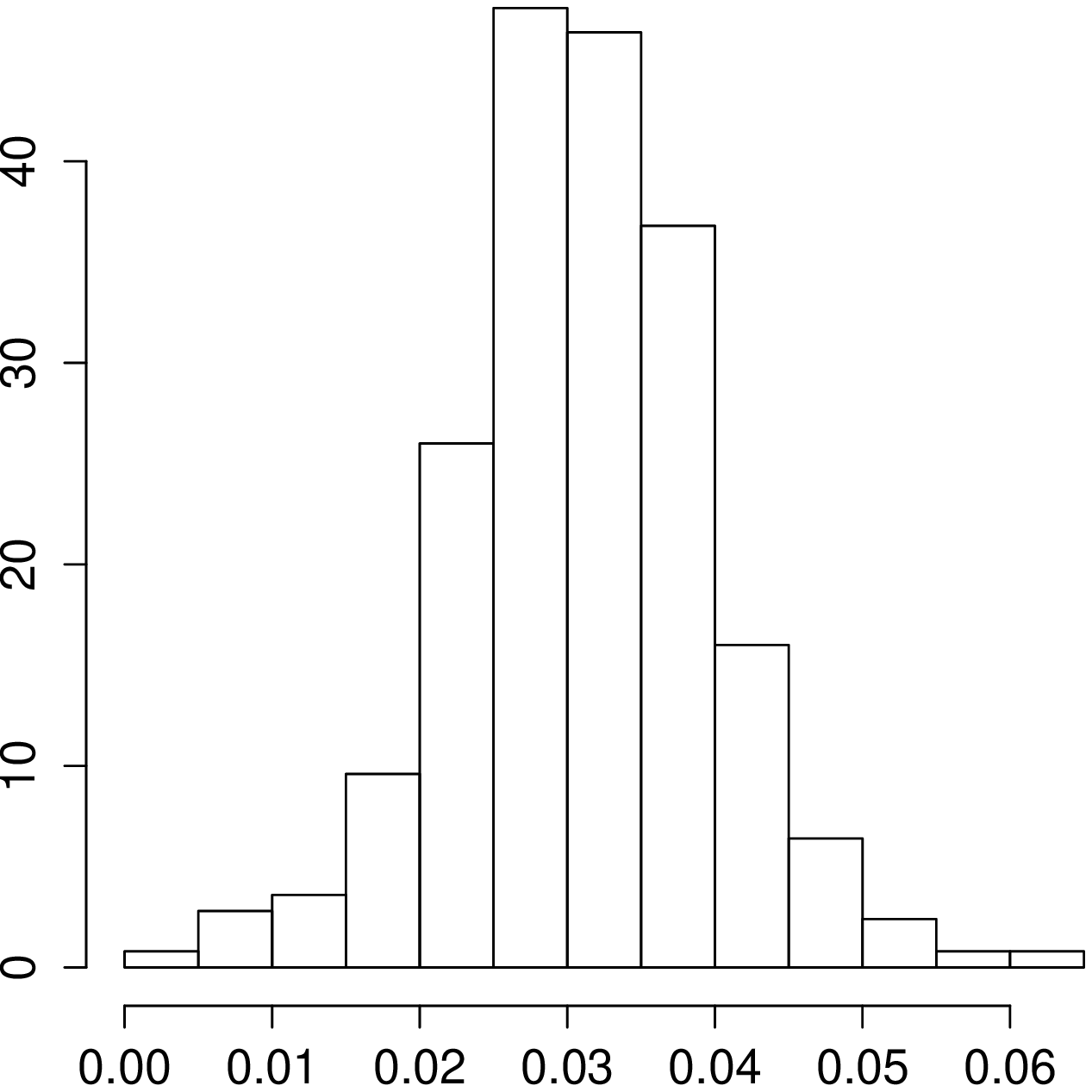}
\end{tabular}
\caption{Histograms of  the estimations of $\alpha=0.03$ from 500 realizations of  DPPs with kernel~\eqref{gaussian kernel intro} on $[0,1]^2$. From left to right :  MCE  \eqref{definition thetan}  based on $K$, MCE \eqref{definition thetan}  based on $g$ and maximum likelihood estimator.}\label{figure histogram DPP}
\end{center}
\end{figure}

\section{Asymptotic properties of minimum contrast estimators based on $K$ and $g$}\label{section special case minimum contrast}

\subsection{Setting}\label{setting}

In the next sections we study the asymptotic properties of \eqref{definition thetan} when $J=K$ and $J=g$, respectively.
The asymptotic is to be understood in the following way. We assume to observe one realization of $\X$ on $D_n$ and we let $D_n$ to expand to $\R^d$ as detailed below. We denote by $\partial D_n$  the boundary of $D_n$.
\begin{definition}\label{definition regular set}
 A sequence of subsets $\lbrace D_n\rbrace_{n\in \N}$ of $\R^d$ is called regular if for all $n\in\N$, $D_n\subset D_{n+1}$, $D_n$ is compact, convex and there exist constants $\alpha_1$ and $\alpha_2$ such that
 \begin{align*}
  \alpha_1 n^d &\leq \Dn \leq \alpha_2 n^d, \\
  \alpha_1 n^{d-1} &\leq \mathcal{H}_{d-1} \left( \partial D_n \right) \leq \alpha_2 n^{d-1}
 \end{align*}
 where $\mathcal{H}_{d-1}$ is the $(d-1)$-dimensional Hausdorff measure.
\end{definition}

Henceforth, we consider the estimator \eqref{definition thetan} under the setting of Section~\ref{estimation} where $\lbrace D_n\rbrace_{n\in \N}$ is a sequence of regular subsets of $\R^d$. 
Moreover, for any $\rho>0$ and $\theta\in \Theta_{\rho}$, we assume that the correlation function associated to $C_{\rho,\theta}$, denoted by  $R_\theta$, does not depend on $\rho$ but only on $\theta$, i.e.  $R_\theta = C_{\rho,\theta}/\rho$. 
Note that this is the case for all parametric families considered in~\cite{lavancierpublish} and \cite{bisciobernoulli}, including the  Whittle-Matèrn, the generalized Cauchy and the generalized Bessel families.

 For $r>0$, we denote by $\Theta_{\rho_0}^{\oplus r}:=  \Theta_{\rho_0} + \overline B(0,r)$ the $r$-dilation  of $\Theta_{\rho_0}$, where $\overline B(0,r)$ denotes the closed ball centred at $0$ with radius $r$. Further, for all $x\in\R^d$, denote $R^{(1)}_\theta(x)$ and $R^{(2)}_\theta(x)$, the gradient, respectively the Hessian matrix, of $R_\theta(x)$  with respect to $\theta$. We make the following assumptions. Specific additional hypotheses in the case $J=K$ and $J=g$ are described in the respective sections. 

\begin{enumerate}[label=($\mathcal{H}$\arabic*)]
 \item For all $\rho>0$, $\Theta_{\rho}$ is a compact convex set with non-empty interior and the mapping $\rho \rightarrow \Theta_\rho$ is continuous with respect to the Haussdorff distance on the compact sets. 
 \label{assumption general domaine}

 \item For all $\theta \in \Theta_{\rho_0}$, $C_{\rho_0,\theta}$ verifies the condition $\K(\rho_0)$ and there exists $\epsilon>0$ such that for all $\theta \in \Theta_{\rho_0}^{ \oplus \epsilon}$, $C_{\rho_0,\theta} \in L^2(\R^d)$ and $\F(C_{\rho_0,\theta})\geq 0$. \label{assumption general Krho}

  \item There exists $\epsilon>0$ such that for all $x\in B(0,\rM)$, the function $\theta \mapsto R_\theta(x)$ is of class $\mathcal{C}^2$ on $\Theta_{\rho_0}^{ \oplus \epsilon} $.
  Further, for $i\in \lbrace 1,2\rbrace$, there exists $M>0$ such that for all $x\in B(0,\rM)$ and $\theta\in \Theta_{\rho_0}^{ \oplus \epsilon} $, $\big|R_{\theta}^{(i)}(x)\big| \leq M$.  \label{assumption general differentiabilty}
\end{enumerate}

The first assumption is needed to handle the fact that the minimisation \eqref{definition thetan} is done over the random set $\Theta_{\hat \rho_n}$ in place of $\Theta_{\rho_0}$.  The two other assumptions deal with the regularity of the kernel with respect to the parameters.

\subsection{MCE based on $K$}\label{section theorem ripley}
Since for  any $\rho>0$ and $\theta \in \Theta_{\rho}$ $R_\theta = C_{\rho,\theta}/\rho$ is assumed to  not depend on $\rho$, the $K$-function \eqref{defK} of $DPP(C_{\rho,\theta})$ does not depend on $\rho$. Consequently we denote it by $K(.,\theta)$. 
 For all $t\geq 0$ and $n\in \N$, we consider the estimator of the $K$-function, see~\cite[Chapter 4]{mollerstatisticalinference},\begin{align}\label{formula estimator ripley}
   \Kn(t) := \frac{1}{\wrho{}\hspace*{-1.5mm}^2} \sum_{(x,y) \in \X^2}^{\neq} \1_{\lbrace x\in D_n \rbrace} \1_{\lbrace y \in D_n^{\circleddash t} \rbrace}   \ \frac{\1_{\lbrace |x-y|\leq t \rbrace}}{|D_n^{\circleddash t}|}
\end{align}
where $\wrho$ is as in~\eqref{estimation rho} and for  $t\geq 0$, $D_n^{\circleddash t} : = \left\lbrace x\in D_n , \ B(x,t) \in D_n\right\rbrace $. 

For all $t\in [\rm,\rM]$, denote by $K^{(1)}(t,\theta)$ and $K^{(2)}(t,\theta)$ the gradient and the Hessian matrix of $K(t,\theta)$ with respect to $\theta$.  We consider the following assumptions.

\begin{enumerate}[label=($\mathcal{H}_K$\arabic*)]
 \item $w$ is a positive and integrable function in $[\rm,\rM]$. \label{assumption ripley w}
  \item If $\rm =0$, then $c\geq 2$. \label{assumption ripley r=0}
    \item  For $\theta_1 \neq \theta_2$, there exists a set $A\in [\rm,\rM]$ of positive Lebesgue measure such that
\begin{align*}
\int_{ x \in B(0,t) } R_{\theta_1}(x)^2 dx \neq \int_{ x \in B(0,t) } R_{\theta_2}(x)^2 dx, \quad \forall t\in A.
\end{align*} \label{assumption ripley identifiability}
  \item The matrix $\int_\rm^\rM w(t) K(t,\theta_0)^{2c-2} K^{(1)}(t,\theta_0) K^{(1)}(t,\theta_0)^Tdt$ is invertible. \label{assumption ripley inversibilite pour B}
\end{enumerate} 

Assumption \ref{assumption ripley w} is not restrictive. The constraint on $c$ implied by~\ref{assumption ripley r=0} in the case $\rm=0$ tends to confirm the practice, which consists in the choice $\rm>0$. \ref{assumption ripley identifiability} is an identifiability assumption and~\ref{assumption ripley inversibilite pour B} turns out to be the main technical assumption.
Define for all $t\in [\rm,\rM]$,
\begin{align*}
 j_K(t)&:= w(t) K(t,\theta_0)^{2c-2} K^{(1)}(t,\theta_0).
\end{align*}
The following theorem states the strong consistency and the asymptotic normality of the MCE based on $K$ for stationary DPPs. It is proved in Section~\ref{section preuve theoreme ripley}.

\begin{theorem}\label{theorem ripley normalite asymptotic}
Let $\X$ be a DPP with kernel $C_{\rho_0,\theta_0} = \rho_0 R_{\theta_0}$ for a given $\rho_0>0$ and $\theta_0$ an interior point of $\Theta_{\rho_0}$. For all $n\in \N$, let $U_n$ be defined as in~\eqref{expression Un discrepancy measure} with $J=K$ and $\Jn=\Kn$. Assume that \ref{assumption general domaine}-\ref{assumption general differentiabilty} and \ref{assumption ripley w}-\ref{assumption ripley inversibilite pour B} hold. Then, the minimum contrast estimator $\widehat{\theta}_n$ defined by~\eqref{definition thetan} exists and is strongly consistent for $\theta_0$.
Moreover, it satisfies
 \begin{align*}
 \sqrt{\Dn}(\widehat{\theta}_n -\theta_0) \convl \mathcal{N} \left[ 0,B_{\theta_0}^{-1} \Sigma_{\rho_0,\theta_0} \lbrace B_{\theta_0}^{-1}  \rbrace^T \right]
\end{align*}
with
\begin{align}
 B_{\theta_0}:= \int_\rm^\rM w(t) K(t,\theta_0)^{2c-2} K^{(1)}(t,\theta_0) K^{(1)}(t,\theta_0)^T dt
\end{align}
and
\begin{align*}
 \Sigma_{\rho_0,\theta_0} =  \int_\rm^\rM \int_\rm^\rM h_{\rho_0,\theta_0} (t_1,t_2)  j_K(t_1) j_K(t_2)  dt_1 dt_2
\end{align*}
where $h_{\rho_0,\theta_0} $ can be expressed in terms of $C_{\rho_0,\theta_0}$. Specifically, for all $(t_1,t_2) \in [\rm,\rM]^2$,
\begin{align*}
h_{\rho_0,\theta_0} (t_1,t_2)&:=2\int_{\R^d} \1_{\lbrace 0< |x| \leq  t_1\rbrace}\1_{\lbrace 0< |x| \leq  t_2\rbrace} \left( c_{[2]}^{red}(x) + \rho_0^2 \right)dx \\
&+ 4\int_{\R^{2d}} \1_{\lbrace 0< |x| \leq  t_1\rbrace}\1_{\lbrace 0< |y-x| \leq  t_2\rbrace} \left( c_{[3]}^{red}(x,y)+ \rho_0 c_{[2]}^{red}(y) \right) dx dy \\
&+ 4\rho_0  \int_{\R^{2d}} \1_{\lbrace 0< |x| \leq  t_1\rbrace} \1_{\lbrace 0< |y| \leq  t_2\rbrace} \left( 2 c_{[2]}^{red}(y) + \rho_0^2 \right) dx dy \\
&+ \int_{\R^{3d}} \1_{\lbrace 0< |x| \leq  t_1\rbrace}\1_{\lbrace 0< |z-y| \leq  t_2\rbrace} c_{[4]}^{red}(x,y,z) dxdydz \\ 
&+ 4\rho_0 \int_{\R^{3d}} \1_{\lbrace 0< |x| \leq  t_1\rbrace} \1_{\lbrace 0< |z-y| \leq  t_2\rbrace} c_{[3]}^{red} (y,z) dxdydz \\
&+ 2 \int_{\R^{3d}} \1_{\lbrace 0< |x| \leq  t_1 \rbrace} \1_{\lbrace 0< |x+z-y| \leq  t_2 \rbrace} c_{[2]}^{red}(y) c_{[2]}^{red} (z) dx dy dz \\
&+ 4\rho_0^2 \int_{\R^{3d}} \1_{\lbrace 0< |x| \leq  t_1\rbrace}  \1_{\lbrace 0< |z-y| \leq  t_2\rbrace} c_{[2]}^{red}(y) dx dy dz \\ 
&- 4\rho_0 \int_{\R^{2d}} \1_{\lbrace 0< |x| \leq  t_1\rbrace} K(t_2,\theta_0) \left( c_{[3]}^{red}(x,y) + 2\rho_0 c_{[2]}^{red}(y)\right) dx dy \\
&- 8 \rho_0 \int_{\R^d} \1_{\lbrace 0< |x| \leq  t_1\rbrace} K(t_2,\theta_0) \left( c_{[2]}^{red}(x) + \rho_0^2 \right) dx \\
&+ 4\rho_0^2 K(t_1,\theta_0)K(t_2,\theta_0) \left(\rho_0-  \int_{\R^d} C_{\rho_0,\theta_0}(x)^2 dx \right)
 \end{align*}
where $c_{[2]}^{red}, c_{[3]}^{red}$ and $c_{[4]}^{red}$ are given with respect to $C_{\rho_0,\theta_0}$ in \eqref{expression densite cumulant 2}-\eqref{expression densite cumulant 4}.
\end{theorem}

Let us notice that the finiteness of the integrals involved in the last expression follows from the Brillinger mixing property of the DPPs with kernel verifying the condition $\K(\rho_0)$, see \cite{bisciobrillingerTCL}.

\subsection{MCE based on $g$}\label{section theoreme pcf}

We assume in this section that all DPPs of the parametric family are isotropic, which is the usual practice when dealing with the pair correlation function.  In this case, for all $\rho>0$ and $\theta \in \Theta_\rho$, there exists $\widetilde{R}_\theta$ such that $R_\theta(x) = \widetilde{R}_\theta(|x|)$ for all $x\in\R^d$ so  that the pcf of $DPP(C_{\rho,\theta})$ writes
\begin{align}\label{expression pcf noyau DPP}
g(x,\theta)=1 - \widetilde{R}_\theta(|x|)^2 =:\tilde{g}(|x|,\theta)
\end{align}
and does not depend on $\rho$. 
 In the following, to alleviate the notation, we omit the symbol tilde and for all $\theta\in\Theta_\rho$, we consider that the domain of definition of  $R_\theta(.)$ and  $g(.,\theta)$ is $\R^+$. Moreover, by symmetry we extend this domain to $\R$. Denote, for all $d\geq 2$, the  surface area of the $d$-dimensional unit ball,
\begin{align*}
 \sigma_d := \frac{2\pi^{d/2}}{\Gamma\left(d/2 \right)} .
\end{align*}
For $n\in\N$ and $t>0$, we consider the kernel estimator of $g$, see~\cite[Section 4.3.5]{mollerstatisticalinference},
\begin{align}\label{formula estimator pcf}
  \widehat{g_n}(t) := \frac{1}{\sigma_d t^{d-1} \wrho{}\hspace*{-1.5mm}^2} \sum_{ (x,y) \in \X^2  }^{\neq} \1_{\left\lbrace x\in D_n,\ y \in D_n\right\rbrace}   \ \frac{1}{b_n |D_n\cap D_n^{ x-y}|} k\left(\frac{t-|x-y|}{b_n}\right)
\end{align}
where  for any $z\in\R^d$ $D_n^{z} : = \left\lbrace u , u+z \in D_n \right\rbrace$, $\wrho$ is as in~\eqref{estimation rho}  and  $b_n$ and $k$ are the bandwidth and the kernel to be chosen according to the assumptions below. 
For all $t\in[\rm,\rM]$, denote by $g^{(1)}(t,\theta)$ and $g^{(2)}(t,\theta)$ the gradient and the Hessian matrix  of $g$ with respect to $\theta$. We consider the assumptions:

\begin{enumerate}[label=($\mathcal{H}_g$\arabic*)]
  \item $\rm>0$. \label{assumption pcf rm}
  
      \item $w$ is a positive and continuous function on $[\rm,\rM]$. \label{assumption pcf w}
      
 \item The kernel $k$ is positive, symmetric and bounded with compact support included in $[-T,T]$ for a given $T>0$. Further, $\int_\R k(x) dx =  1$. \label{assumption pcf k} 
 \item $\{b_n\}_{ n\in\N }$ is a positive sequence, $b_n\rightarrow 0$, $b_n \Dn \rightarrow +\infty $ and $b_n^4 \Dn \rightarrow 0$. \label{assumption pcf bn} 
 \item There exists $\epsilon>0$ such that for all $\theta \in \Theta_{\rho_0}^{ \oplus \epsilon}$, $R_\theta(.)$ is of class $\mathcal{C}^2$ on $\R \setminus \lbrace 0\rbrace$. \label{assumption pcf differentiabilty}
\item  For $\theta_1 \neq \theta_2$, there exists a set $A\in [\rm,\rM]$ of positive Lebesgue measure such that
\begin{align*}
\left| R_{\theta_1}(t)\right| \neq  \left| R_{\theta_2}(t) \right|, \quad \forall t\in A.
\end{align*} \label{assumption pcf identifiability}
 \item The matrix $\int_\rm^\rM w(t) g(t,\theta_0)^{2c-2} g^{(1)}(t,\theta_0) g^{(1)}(t,\theta_0)^T dt$ is invertible. \label{assumption pcf B inversible}
\end{enumerate}
The first four assumptions are easy to satisfy by appropriate choices of $\rm$, $w$, $b_n$ and $k$.  \ref{assumption pcf differentiabilty} is not restrictive and is satisfied by all parametric families considered  in~\cite{lavancierpublish} and \cite{bisciobernoulli}. \ref{assumption pcf identifiability} is an identifiability assumption and as in the previous section, the main technical assumption is in fact~\ref{assumption pcf B inversible}.
The proof of the following theorem is postponed to Section~\ref{section preuve theoreme pcf}. Put
\begin{align*}
j_g(t)&:= w(t) g(t,\theta_0)^{2c-2} g^{(1)}(t,\theta_0), \quad t\in [\rm,\rM].
\end{align*}
\begin{theorem}\label{theorem pcf normalite asymptotique}
Let $\X$ be an isotropic DPP with kernel $C_{\rho_0,\theta_0} = \rho_0 R_{\theta_0}$ for a given $\rho_0>0$ and $\theta_0$ an interior point of $\Theta_{\rho_0}$. For all $n\in \N$, let $U_n$ be defined as in~\eqref{expression Un discrepancy measure} with $J=g$ and $\Jn=\gn$. Assume that \ref{assumption general domaine}-\ref{assumption general differentiabilty} and \ref{assumption pcf rm}-\ref{assumption pcf B inversible} hold. Assume further that for all $\theta\in\Theta_{\rho_0}$, $R_\theta(.)$ is isotropic. Then, the minimum contrast estimator $\widehat{\theta}_n$ defined by~\eqref{definition thetan} exists and is consistent for $\theta_0$.
Moreover, it satisfies
 \begin{align*}
 \sqrt{\Dn}(\widehat{\theta}_n -\theta_0) \convl \mathcal{N} \left[ 0,B_{\theta_0}^{-1} \Sigma_{\rho_0,\theta_0} \lbrace B_{\theta_0}^{-1}  \rbrace^T \right]
\end{align*}
with
\begin{align*}
 B_{\theta_0}:= \int_\rm^\rM w(t) g(t,\theta_0)^{2c-2} g^{(1)}(t,\theta_0) g^{(1)}(t,\theta_0)^T dt
\end{align*}
and
 \begin{align*} 
  \Sigma_{\rho_0,\theta_0} &= 2\int_{\R^d} \1_{\lbrace \rm \leq |x| \leq \rM \rbrace } \frac{j_g(|x|) j_g(|x|)}{\sigma_d^2 |x|^{2(d-1)}} \left( c_{[2]}^{red} (x) + \rho_0^2 \right) dx \\
         &+ 4\int_{\R^{2d}} \1_{\lbrace \rm \leq |x|, |y-x|  \leq \rM \rbrace } \frac{ j_g(|x|) j_g(|y-x|)}{\sigma_d^2 |x|^{d-1}|y-x|^{d-1}} \left( c_{[3]}^{red}(x,y) + \rho_0 c_{[2]}^{red}(y)  \right) dxdy\\
         &+4\rho_0 \int_{\R^{2d}}  \1_{\lbrace \rm \leq |x|,|y| \leq \rM \rbrace }\frac{ j_g(|x|)  j_g(|y|)}{\sigma_d^2 |x|^{d-1}|y|^{d-1}} \left( 2c_{[2]}^{red}(x) + \rho_0^2 \right)  dx dy \\
         &+  \int_{\R^{3d}}  \1_{\lbrace \rm \leq |x|,|z-y|  \leq \rM \rbrace } \frac{ j_g(|x|) j_g(|z-y|)}{\sigma_d^2 |x|^{d-1}|z-y|^{d-1}} c_{[4]}^{red} (x,y,z) \ dx dy dz \\
         &+ 4\rho_0 \int_{\R^{3d}}  \1_{\lbrace \rm \leq |x|,|z-y|  \leq \rM \rbrace } \frac{ j_g(|x|) j_g(|z-y|)}{\sigma_d^2 |x|^{d-1}|z-y|^{d-1}}  c_{[3]}^{red} ( y,z)\   dx dy dz  \\
         &+ 2\int_{\R^{3d}}  \1_{\lbrace \rm \leq |x|,  |z-y+x|  \leq \rM \rbrace } \frac{ j_g(|x|) j_g(|z-y+x|)}{\sigma_d^2 |x|^{d-1}|z-y+x|^{d-1}} c_{[2]}^{red} (y) c_{[2]}^{red} (z)  \ dx dy dz  \\
         &+ 4\rho_0^2 \int_{\R^{3d}}  \1_{\lbrace \rm \leq |x|,  |z-y|  \leq \rM \rbrace } \frac{ j_g(|x|) j_g(|z-y|)}{\sigma_d^2 |x|^{d-1}|z-y|^{d-1}}  c_{[2]}^{red} (y)  \ dx dy dz \\ 
         &- 4\rho_0 \left( \int_{\rm}^{\rM}g(t,\theta_0)  j_g(t) dt \right) \\
         & \hspace*{2.85cm}  \int_{\R^{2d}} \1_{\lbrace \rm \leq |x| \leq \rM \rbrace } \frac{j_g(|x|)}{\sigma_d |x|^{d-1}}  \left(  c_{[3]}^{red}(x,y) + 2\rho_0 c_{[2]}^{red}(y)  \right) dxdy \\ 
                  &-  8\rho_0 \left( \int_{\rm}^{\rM}g(t,\theta_0)  j_g(t) dt \right) \int_{\R^{d}} \1_{\lbrace \rm \leq |x| \leq \rM \rbrace }\frac{j_g(|x|)}{\sigma_d |x|^{d-1}} \left( c_{[2]}^{red}(x) +\rho_0^2 \right) dx \\
         &+ 4\rho_0^2 \left( \int_{\rm}^{\rM} g(t,\theta_0) j_g(t)dt \right)^2 \left( \rho_0 -\int_{\R^d} C_{\rho_0,\theta_0}(x)^2 dx \right) \end{align*}
where $c_{[2]}^{red}, c_{[3]}^{red}$ and $c_{[4]}^{red}$ are given  in \eqref{expression densite cumulant 2}-\eqref{expression densite cumulant 4}.
\end{theorem}

\section{Proofs}\label{proofs}

\subsection{Proof of Theorem~\ref{theorem ripley normalite asymptotic}}\label{section preuve theoreme ripley}

Since $C_{\rho_0,\theta_0}$ verifies $\K(\rho_0)$, $\wrho$ converges almost surely to  $\rho_0$, so by~\ref{assumption general domaine}, for all $\epsilon>0$, there exists $N\in \N$ such that for all $n\geq N$,  $\Theta_{\hat \rho _n} \subset \Theta_{\rho_0}^{ \oplus \epsilon}$ almost surely. Henceforth, without loss of generality, we let $\epsilon >0$ and assume that $\Theta_{\hat \rho _n} \subset \Theta_{\rho_0}^{ \oplus \epsilon}$ for all $n\in \N$. 
We apply below the general Theorems~\ref{theoreme consistency}-\ref{theoreme normalite asymptotique} of the appendix to prove that the estimator  $\ttheta$ defined in~\eqref{definition thetatilde}  with $\Theta =\Theta_{\rho_0}^{ \oplus \epsilon}$, $J=K$ and  $\Jn=\Kn$ is consistent and asymptotically normal. 
 As a consequence, almost surely, there exist $r>0$ such that $B(\theta_0,r) \subset \Theta_{\rho_0}$ and $N_r \in \N$ such that for all $n\geq N_r$,  $\ttheta \in B(\theta_0,r)$. From Lemma~\ref{lemma boule incluse dans grossisement} in the appendix and \ref{assumption general domaine}, we deduce that   for $n$ sufficiently large, $B(\theta_0,r) \subset \Theta_{\hat \rho _n}$. Hence, almost surely, for $n$ large enough, the minimum of $U_n$ is attained in $\Theta_{\hat \rho _n} \subset \Theta_{\rho_0}^{ \oplus \epsilon}$ so that  $\ttheta$ in \eqref{definition thetatilde} and $\wtheta$ in \eqref{definition thetan} coincide.
 
 Let us now prove the strong consistency and asymptotic normality of $\ttheta$ in \eqref{definition thetatilde} when $\Theta = \Theta_{\rho_0}^{ \oplus \epsilon}$,  $J=K$ and  $\Jn=\Kn$. To that end, we verify all the assumptions of Theorems~\ref{theoreme consistency}-\ref{theoreme normalite asymptotique}. 
The general setting in Section~\ref{setting},  Assumptions~\ref{assumption general domaine} and \ref{assumption ripley w} imply directly \ref{assumption Theta compact et Dn regular}-\ref{assumption w}.
  For all $\theta \in \Theta$, we have
\begin{align}\label{formula ripley DPP}
 K(t,\theta) = \sigma_d t^d - \int_{ x \in B(0,t) } R_\theta(x)^2 dx
\end{align}
where $\F(R_\theta) \geq 0$ by~\ref{assumption general Krho}. Further, by \cite[Corollary 1.4.13]{sasvari2013multivariate}, for all $\theta \in \Theta$, if for a given $x\neq 0$, $|R_\theta(x)| =1$, then $R_\theta$ is invariant by translation of $x$. Since for all $\theta \in \Theta$, $R_\theta( . ) \in L^2(\R^d)$, this is impossible so, for all $x\neq 0$ and $\theta \in \Theta$, $|R_\theta(x)| <1$. Hence, by~\eqref{formula ripley DPP}, $K(t,\theta)>0$ on $ (\rm,\rM]\times \Theta$ and $K(.,.)$ is continuous on $ [\rm,\rM]\times \Theta$.
Consequently,  $K(.,.)^c$ is continuous for all $c\in \R$ if $\rm>0$ and for all $c>0$ if $\rm=0$. Therefore, under~\ref{assumption general domaine}-\ref{assumption general differentiabilty}  and \ref{assumption ripley r=0}, \ref{assumption continuite de J et param c} holds. By the same arguments, $K(.,.)^{c-2}$ and $K(.,.)^{2c-2}$ are continuous for all $c\in \R$ if $\rm>0$ and for all $c\geq 2$ if $\rm =0$. Thus
\ref{assumption integrability pour B} holds. For all $t\in[\rm,\rM]$, $\Kn(t)$ is bounded by $\Kn(\rM)$ and it follows from the ergodic theorem that $\Kn(\rM)$ is almost surely finite as soon as $n$ and so $D_n$ is large enough. 
 Moreover, by Lemma~\ref{lemma conv unif Ripley function}, $\Kn(t)$ is almost surely strictly positive for $t>0$ and $n$ large enough. 
Hence, under~\ref{assumption general domaine}-\ref{assumption general differentiabilty} and~\ref{assumption ripley r=0}, \ref{assumption Jn borne positive} holds.
We have for all $\theta\in \Theta$ and $t\in(0,\rM)$
 \begin{align*}
  K^{(1)}(t,\theta)= -\frac{\partial}{\partial\theta} \int_{x\in B(0,t)}  R_\theta(x)^2 dx.
 \end{align*}
By \ref{assumption general differentiabilty}, the function $(x,\theta) \mapsto R^{(1)}_\theta(x)$ is continuous with respect to $\theta$ and bounded for all $x\in B(0,\rM)$ and $\theta\in\Theta$. Thus 
by the dominated convergence theorem,
\begin{align}\label{expression J1 ripley}
   K^{(1)}(t,\theta)= - 2\int_{x\in B(0,t)} R_\theta(x)  R^{(1)}_\theta(x) dx.
\end{align}
We obtain similarly
\begin{align*}
   K^{(2)}(t,\theta)= -2 \int_{x\in B(0,t)} \left(  R^{(1)}_\theta(x)R^{(1)}_\theta(x)^T +R^{(2)}_\theta(x) R_\theta(x)   \right)  dx.
\end{align*}
 By \ref{assumption general differentiabilty}, the terms inside the integral in the last equation are bounded uniformly with respect to $(x,\theta) \in B(0,\rM) \times \Theta$. Therefore, $K^{(1)}(t,\theta)$ and $K^{(2)}(t,\theta)$ are continuous with respect to $\theta$ and uniformly bounded with respect to $t\in [\rm,\rM]$ and $\theta\in \Theta$ so \ref{assumption J derivable 2} holds. 
Assumptions \ref{assumption identifiability} and  \ref{assumption B inversible} are directly implied by \ref{assumption ripley identifiability} and \ref{assumption ripley inversibilite pour B}, respectively.
 The assumption $(\mathcal{A}5)'$ is proved by Lemma~\ref{lemma conv unif Ripley function} below,
while Lemmas~\ref{lemma bornitude}-\ref{lemma variance limite ripley} are preliminary results for Lemma~\ref{lemma ripley TCL}  which proves the remaining assumption $(\mathcal{TCL})$.

\begin{lemma}\label{lemma conv unif Ripley function}
Let $K$ be the Ripley's $K$-function of a DPP with kernel $C$ verifying $\K(\rho_0)$ and $\Kn$ the estimator
given by~\eqref{formula estimator ripley}. Then, for all $\rM > \rm \geq 0$, 
 \begin{align*}
 \sup_{t\in  [\rm,\rM]} \left| \Kn(t) - K(t) \right| \convPs 0,
 \end{align*}
\end{lemma}

\begin{proof}
Since a stationary DPP is ergodic by~\cite[Theorem 7]{Soshnikov:00}, we have
\begin{align}\label{convergence ps intensite}
 \wrho \convPs \rho_0
\end{align}
and
\begin{align}\label{convergence ps ripley}
 \sup_{ t \in [\rm,\rM]} \left| \wrho^{\, 2} \Kn(t) - \rho_0^2 K(t) \right| \convPs 0,
\end{align}
see for instance \cite[Section 4.2.2]{heinrichasymptotic:13}. Further, as $K$ is an increasing function, we have
\begin{multline*}
\wrho^{\, 2} \sup_{ t \in [\rm,\rM]} \left|  \Kn(t) -  K(t) \right| \\ \leq  \sup_{ t \in [\rm,\rM]} \left| \wrho^{\, 2} \Kn(t) - \rho_0^2 K(t) \right| +  K(\rM)\sup_{ t \in [\rm,\rM]} \left| \wrho^{\, 2}  - \rho_0^2\right|.
\end{multline*}
Hence, by \eqref{convergence ps intensite}-\eqref{convergence ps ripley} and the last equation, we have the convergence 
\begin{align*}
  \sup_{ t \in [\rm,\rM]} \left|  \Kn(t) -  K(t) \right| \convPs 0.
\end{align*}
\end{proof}

  \begin{lemma}\label{lemma bornitude}
   If \ref{assumption ripley w}-\ref{assumption ripley r=0} and \ref{assumption general differentiabilty} hold, then for all $\rM > \rm \geq 0$,
   \begin{align*}
     \int_\rm^\rM \left|j_K(t) \right| dt < +\infty.
   \end{align*}
 \end{lemma}
 
\begin{proof}

By  \eqref{expression J1 ripley}, we have
 \begin{align}\label{integrabilite ?}
  \int_\rm^\rM \left|j_K(t) \right| dt = 2\int_\rm^\rM \left|w(t)   K(t,\theta_0)^{2c-2} \int_{x\in B(0,t)} R_{\theta_0}(x)  R^{(1)}_{\theta_0}(x) dx \right|dt.
 \end{align}
By \ref{assumption general differentiabilty}, the function defined for all $t\geq 0$ by
\begin{align*}
t \mapsto \int_{x\in B(0,t)} R_{\theta_0}(x)  R^{(1)}_{\theta_0}(x) dx
\end{align*}
is continuous so bounded on $[\rm,\rM]$. As already noticed after~\eqref{formula ripley DPP}, $K(t,\theta)>0$ on $ (\rm,\rM]\times \Theta$. Consequently, if $\rm>0$, the lemma is proved since $w$ is integrable on $[\rm,\rM]$ by~\ref{assumption ripley w}. Finally, if $\rm=0$, the integrability at $0$ of the function $t \mapsto \left| j_K(t) \right|$ follows from~\ref{assumption ripley r=0}.
\end{proof}

To shorten, define for all $n\in\N$ and $t\in [\rm,\rM]$, 
\begin{align*}
  H_n(t)&:=\wrho^{\, 2} \Kn(t)- 2\rho_0 K(t,\theta_0) \wrho .
\end{align*}
\begin{lemma}\label{lemma variance limite ripley}
If \ref{assumption general domaine}-\ref{assumption general differentiabilty} and \ref{assumption ripley w} hold, for all $s\in\R^d$, we have
 \begin{align*}
\lim_{n\rightarrow +\infty} \Dn \mathrm{Var} \Big( \int_\rm^\rM H_n(t)   s^T j_K(t) dt \Big)=  \int_{[\rm,\rM]^2}   h(t_1,t_2) s^T j_K(t_1) s^T j_K(t_2)  dt_1 dt_2
\end{align*}
where  $h_{\rho_0,\theta_0} $ is defined as in Theorem~\ref{theorem ripley normalite asymptotic}.
\end{lemma}

\begin{proof}
From~\eqref{formula estimator ripley}, we have
\begin{align*}
\int_\rm^\rM H_n(t)   s^T j_K(t) dt = \sum_{(x,y) \in \X^2} f_n (x,y) -   \sum_{x\in\X} h_n(x) 
\end{align*}
where for all $n\in\N$,
\begin{align*}
 f_n(x,y):=   \1_{\lbrace x\in D_n \rbrace} \int_\rm^\rM \frac{1}{|D_n^{\circleddash t}|} \1_{\lbrace y \in D_n^{\circleddash t} \rbrace} \1_{\lbrace 0<|x-y|\leq t \rbrace}   s^T j_K(t) dt 
 \end{align*}
 and 
 \begin{align*}
  h_n(x) = \frac{2\rho_0}{\Dn} \1_{\lbrace x\in D_n \rbrace}  \int_\rm^\rM K(t,\theta_0)    s^T j_K(t) dt.
 \end{align*}
 Notice that for all $n\in\N$ and $x\in\R^d$, $f_n(x,x)=0$. Thus, we have from the last equation,
\begin{multline*}
 \mathrm{Var} \left( \int_\rm^\rM H_n(t)   s^T j_K(t) dt \right) \\ = \mathrm{Var} \left(  \sum_{(x,y) \in \X^2}^{\neq} f_n (x,y) \right)  + \mathrm{Var} \left(  \sum_{x\in\X} h_n(x) \right) - 2\, \mathrm{Cov} \left(  \sum_{(x,y) \in \X^2}^{\neq} f_n (x,y) ,  \sum_{x\in\X} h_n(x)  \right).
\end{multline*}
These terms are developed in  Lemmas~7.1-7.3 of~\cite{bisciobrillingerTCL}, whereby we deduce the limit by a long but straightforward calculus. 
\end{proof}

\begin{lemma}\label{lemma ripley TCL}
 If \ref{assumption general domaine}-\ref{assumption general differentiabilty} and \ref{assumption ripley w}-\ref{assumption ripley r=0} hold, then
   \begin{align*}
\sqrt{\Dn}\int_\rm^\rM  \left[ \Kn(t) - K(t,\theta_0) \right] j_K(t)dt \convl \mathcal{N} (0,\Sigma_{\rho_0,\theta_0})                                                                                           \end{align*}
 where $\Sigma_{\rho_0,\theta_0}$ is defined as in Theorem~\ref{theorem ripley normalite asymptotic}.
\end{lemma}

\begin{proof}
For all $n\in\N$, we have
\begin{multline}\label{egalite tcl ripley 1}
 \rho_0^2 \sqrt{\Dn} \int_{\rm}^{\rM} \left[ \Kn(t)-K(t,\theta_0) \right] j_K(t) dt  =  \sqrt{\Dn} \int_{\rm}^{\rM} \left[ \rho_0^2-\wrho^{\, 2} \right] \Kn(t) j_K(t) dt \\ +  \sqrt{\Dn} \int_{\rm}^{\rM} \left[ \wrho^{\, 2}\Kn(t)-\rho_0^2 K(t,\theta_0) \right] j_K(t) dt.
\end{multline}
Since $\X$ is ergodic by~\cite[Theorem 7]{Soshnikov:00},  $\wrho$ converges almost surely to $\rho_0$. Then, by Taylor expansion of the function $x\rightarrow x^2$ at $\rho_0$, we have almost surely
\begin{align}\label{taylor sur intensite ripley}
 \left[ \rho_0^2-\wrho^{\, 2} \right] =    2\rho_0  \left[ \rho_0-\wrho \right]  +  o\left(   \rho_0-\wrho \right).
\end{align}
Moreover,
 \begin{multline}\label{terme negligeable preuve tcl ripley}
 2\rho_0 \sqrt{\Dn} \int_{\rm}^{\rM} \left[ \rho_0-\wrho \right] \Kn(t) j_K(t)  dt \\ = 2\rho_0 \sqrt{\Dn} \int_{\rm}^{\rM} \left[ \rho_0-\wrho \right] \left[\Kn(t) - K(t,\theta_0) \right] j_K(t) dt  \\ +   2\rho_0 \sqrt{\Dn}\int_{\rm}^{\rM} \left[ \rho_0-\wrho \right] K(t,\theta_0) j_K(t) dt.
 \end{multline}
Using the notation
\begin{align*}
 A_n &=  2\rho_0 \sqrt{\Dn}\left[ \rho_0-\wrho \right] \int_{\rm}^{\rM}  \Kn(t) j_K(t) dt ,\\
 B_n &=  2\rho_0 \sqrt{\Dn}\left[ \rho_0-\wrho \right]  \int_{\rm}^{\rM} \left[\Kn(t) - K(t,\theta_0) \right] j_K(t) dt,\\
 C_n &=\sqrt{\Dn}\int_{\rm}^{\rM} \left(  \left[ \rho_0-\wrho \right] 2\rho_0 K(t,\theta_0) + \left[ \wrho^{\, 2}\Kn(t)-\rho_0^2 K(t,\theta_0) \right] \right) j_K(t) dt,
\end{align*}
we have by~\eqref{egalite tcl ripley 1}-\eqref{terme negligeable preuve tcl ripley},
\begin{align}\label{egalite tcl ripley short}
  \rho_0^2 &\sqrt{\Dn} \int_{\rm}^{\rM} \left[ \Kn(t)-K(t,\theta_0) \right] j_K(t) dt = B_n + C_n + o\left( A_n \right).
\end{align}
We prove that $B_n + o(A_n)$ tends in probability to $0$ and $C_n$ tends in distribution to a Gaussian variable. Then, the proof is concluded by Slutsky's theorem and \eqref{egalite tcl ripley short}.
By Lemma~\ref{lemma conv unif Ripley function}, 
\begin{align*}
  \sup_{t\in  [\rm,\rM]} \left| \Kn(t) - K(t,\theta_0) \right| \convPs 0
\end{align*}
so
\begin{align}\label{convergence integrale Kn}
 \int_{\rm}^{\rM}  \Kn(t) j_K(t) dt \convPs \int_{\rm}^{\rM}  K(t,\theta_0) j_K(t) dt.
\end{align}
Since $K(.,\theta_0)$ is continuous on $[\rm,\rM]$, $\int_{\rm}^{\rM}  K(t,\theta_0) j_K(t) dt $ is finite by Lemma~\ref{lemma bornitude}. Hence, by Corollary~\ref{corollary convergence normal intensite DPP}, \eqref{convergence integrale Kn} and Slutsky's theorem, we deduce that $B_n \convP 0$ and $o(A_n) \convP 0$.

As to the term $C_n$, notice that
\begin{align}\label{autre egalite Cn tcl ripley}
C_n = \sqrt{\Dn} \left(\int_\rm^\rM H_n(t)   j_K(t)dt -  \left[-  \int_\rm^\rM \rho_0^2 K(t,\theta_0) j_K(t)dt  \right] \right).
\end{align}
We prove the convergence in distribution of $C_n$ by the Cramer-Wold device, see for instance \cite[Theorem 29.4]{billingsleyprobaetmesure1979}. For all $t\in[\rm,\rM]$ and $s\in \R^p$, we have
\begin{align*}
s^T C_n=  \sqrt{\Dn} \left( \int_\rm^\rM  H_n(t)   s^T j_K(t) dt  - \left[-  \int_\rm^\rM \rho_0^2 K(t,\theta_0)  s^T j_K(t)dt  \right] \right).
\end{align*}
By~\eqref{formula estimator ripley}, we have
\begin{align}\label{autre exression fonctionnelle}
\int_\rm^\rM H_n(t)   s^T j_K(t) dt = \sum_{(x,y) \in \X^2} f_{D_n} (x,y)
\end{align}
where
\begin{align*}
 f_{D_n}(x,y):=  \  \1_{\lbrace x\in D_n \rbrace}\int_\rm^\rM \left( \frac{ \1_{\lbrace y \in D_n^{\circleddash t} \rbrace} }{|D_n^{\circleddash t}|} \1_{\lbrace 0<|x-y|\leq t \rbrace}   - 2\rho_0 \frac{K(t,\theta_0)}{\Dn} \1_{\lbrace x-y=0\rbrace}    \right) s^T j_K(t) dt.
\end{align*}
Notice that for $t\in[\rm,\rM]$, $s^T j_K(t)\leq |j_K(t)| |s|$ and $K(t,\theta_0)\leq K(\rM,\theta_0)$ so we have
\begin{multline} \label{majoration fDn}
 \left| f_{D_n}(x,y) \right| \\ \leq  \frac{|s|}{|D_n^{\circleddash \rM}|}  \1_{D_n}(x) \left( \1_{\lbrace 0<|x-y|\leq \rM \rbrace} +\1_{\lbrace x-y=0\rbrace} 2\rho_0  K(\rM,\theta_0)   \right) \int_{\rm}^\rM \left|j_K(t)\right|  dt.
\end{multline}
The right-hand term in~\eqref{majoration fDn} is compactly supported and is bounded by Lemma~\ref{lemma bornitude}.
Moreover, 
\begin{multline*}
 \E\left( \int_\rm^\rM  \left| H_n(t)    s^T j_K(t) \right|dt  \right)  \\ \leq  |s| \left[ \E\left(\left| \wrho^{\, 2} \Kn(t) \right|\right)+ 2\rho_0 K(\rM,\theta_0) \E\left(\left|\wrho \right| \right)\right] \int_\rm^\rM |j_K(t)|dt.
\end{multline*}
 Further, for $n\in \N$ and $t\in[\rm,\rM]$, $\wrho^{\, 2} \Kn(t) $ and $\wrho$ are positive and unbiased estimator of $\rho_0^2 K(t,\theta_0)$ and $\rho_0$, respectively, see for instance \cite[Section 4.2.2]{heinrichasymptotic:13}. Thus,
\begin{align*}
   \E\left( \int_\rm^\rM \left|  H_n(t)  s^T j_K(t)\right| dt  \right) \leq    3|s| \rho_0^2 K(\rM,\theta_0) \int_\rm^\rM |j_K(t)|dt,
\end{align*}
which is finite by Lemma~\ref{lemma bornitude}. Then, by Fubini's theorem, \eqref{autre exression fonctionnelle} and the last equation, we have
\begin{align*}
\E\left( \sum_{(x,y)\in \X^2}  f_{D_n}(x,y) \right) = -  \int_\rm^\rM \rho_0^2 K(t,\theta_0)s^T j_K(t) dt.
\end{align*}
Moreover, by \eqref{autre exression fonctionnelle} and  Lemma~\ref{lemma variance limite ripley}, 
\begin{align*}
 \lim_{n \rightarrow + \infty} Var \left( \sqrt{\Dn} \sum_{(x,y) \in \X^2} f_{D_n} (x,y) \right) = s^T \Sigma_{\rho_0,\theta_0} s. 
\end{align*}
Therefore, by~\eqref{autre egalite Cn tcl ripley}-\eqref{majoration fDn}, the last two equations and Theorem~\ref{TCL jolivet modifie}, we have
\begin{align*}
s^T C_n \convl N(0,s^T \Sigma_{\rho_0,\theta_0} s ).
\end{align*}
which proves that $ C_n \convl N(0,\Sigma_{\rho_0,\theta_0})$.
\end{proof}

\subsection{Proof of Theorem~\ref{theorem pcf normalite asymptotique}}\label{section preuve theoreme pcf}

As in the proof of Theorem~\ref{theorem ripley normalite asymptotic}, we consider without loss of generality  $\epsilon >0$ such that $\Theta_{\hat \rho _n} \subset \Theta_{\rho_0}^{ \oplus \epsilon}$, for all $n\in \N$. 
We  prove below the consistency and asymptotic normality of  $\ttheta$ defined in~\eqref{definition thetatilde}  with $\Theta = \Theta_{\rho_0}^{ \oplus \epsilon}$,  $J=g$ and  $\Jn=\gn$. Then, for $r\geq 0$ such that $B(\theta_0,r) \subset \Theta_{\rho_0}$, we have
\begin{align*}
 P( \ttheta \in B(\theta_0, r) ) \xrightarrow[n\rightarrow +\infty]{} 1.
\end{align*}
Thus, by Lemma~\ref{lemma point dehors cpct}, with probability tending to one $\ttheta \in \Theta_{\hat \rho _n}$ so
\begin{align*}
 P( \ttheta = \wtheta ) \xrightarrow[n\rightarrow +\infty]{} 1.
\end{align*}
Therefore, $\wtheta$ has the same asymptotic behaviour than $\ttheta$.

\medskip
Let us now determine the asymptotic properties of $\ttheta$ by application of Theorems~\ref{theoreme consistency} and~\ref{theoreme normalite asymptotique}. The assumptions \ref{assumption Theta compact et Dn regular}, \ref{assumption w}, \ref{assumption identifiability}, \ref{assumption J derivable 2} and \ref{assumption B inversible} are directly implied by \ref{assumption general domaine}-\ref{assumption general differentiabilty}, \ref{assumption pcf rm}, \ref{assumption pcf w}, \ref{assumption pcf identifiability} and \ref{assumption pcf B inversible}.  Moreover, $\rm>0$ by~\ref{assumption pcf rm} so \ref{assumption Jn borne positive} is directly implied by \eqref{formula estimator pcf}, \ref{assumption pcf k}, \ref{assumption pcf bn} and the ergodic theorem, see \cite{zessinergodic} or~\cite{heinrichasymptotic:13}.
By \ref{assumption general Krho}, $R_{\theta_0}(.)$  is continuous on $[\rm,\rM]$ so is $g$. By \cite[Corollary 1.4.14]{sasvari2013multivariate}, for all $\theta \in \Theta$, if for a given $t>0$, $|R_\theta(t)| =1$, then $R_\theta$ is periodic of period $t$. This is incompatible with~\ref{assumption general Krho} so,  for all $t>0$ and $\theta \in \Theta$, $|R_\theta(t)| <1$. 
Consequently, by \eqref{expression pcf noyau DPP} and~\ref{assumption pcf rm}, $g(t,\theta)$ is strictly positive for all $(t,\theta) \in [\rm,\rM]\times \Theta$. Thus, for all $c\in \R$, $g(.,.)^{c}$ is well defined and strictly positive on $[\rm,\rM]\times \Theta$ so \ref{assumption continuite de J et param c}  holds. By the same arguments, it follows that \ref{assumption integrability pour B} holds. Finally, the assumptions~\ref{assumption convergence uniforme} and $(\mathcal{TCL})$ are proved by Lemmas~\ref{lemma conv univ pcf} and \ref{lemma pcf TCL}, respectively while the other lemmas are auxiliary results.

\begin{lemma}\label{lemma conv univ pcf}
 If \ref{assumption general domaine}-\ref{assumption general differentiabilty},   \ref{assumption pcf rm} and \ref{assumption pcf k}-\ref{assumption pcf bn}  hold then, for all $\rM > \rm >0$, there exists a set $A$ verifying $\left| [\rm,\rM] \setminus A\right| =0$ such that 
  \begin{align*}
\sup_{ t \in A} \left| \gn(t) - g(t,\theta_0) \right| \convP 0.
 \end{align*}
\end{lemma}

\begin{proof}
From~\ref{assumption general Krho}-\ref{assumption general differentiabilty} and  \ref{assumption pcf k}-\ref{assumption pcf bn}  we can use Proposition~4.5 in~\cite{bisciobrillingerTCL} that gives
\begin{multline}\label{egalite application lemma stella}
 \E\left[ \int_\rm^\rM \left(\wrho^{\, 2} \gn(t) - \rho_0^2 g(t,\theta_0) \right)^2 dt \right] \\ = \frac{2 \rho_0^2}{b_n \Dn} \int_\rm^\rM \frac{g(t,\theta_0)}{  \sigma_d t^{d-1} } dt \int_\R k(x)^2 dx + O\left(\frac{1}{\Dn}\right) + O(b_n^4).
\end{multline}
By \ref{assumption pcf rm}, \ref{assumption pcf k} and \ref{assumption general differentiabilty} we have $\int_\rm^\rM \frac{g(t,\theta_0)}{\sigma_d t^{d-1} } dt \int_\R k(x)^2 dx< +\infty$. Hence, with \ref{assumption pcf bn}, the right-hand term in~\eqref{egalite application lemma stella} tends to $0$ as $n$ tends to infinity. Moreover, the term inside the expectation in~\eqref{egalite application lemma stella} is positive so there exists a set $A$  as in Lemma~\ref{lemma conv univ pcf} such that 
 \begin{align}\label{conv unig rhon gn}
  \sup_{ t \in A} \left| \wrho^{\, 2}\gn(t) - \rho_0^ 2 g(t,\theta_0) \right| \convP 0.
 \end{align}
 We have
 \begin{align*}
\wrho^{\, 2} \sup_{ t \in A} \left|  \gn(t) -  g(t,\theta_0) \right|  \leq  \sup_{ t \in A }\left| \wrho^{\, 2} \gn(t) - \rho_0^2 g(t,\theta_0) \right| + \bigg(\sup_{ t \in A} g(t,\theta_0) \bigg)   \left| \wrho^{\, 2}  - \rho_0^2\right|.
\end{align*}
By~\ref{assumption general domaine}-\ref{assumption general Krho}, it follows from Corollary~\ref{corollary convergence normal intensite DPP} that $\wrho$ converges in probability to $\rho_0$. Further, by~\ref{assumption general differentiabilty} and~\eqref{expression pcf noyau DPP}, $g(.,\theta_0)$ is bounded on $[\rm,\rM]$. Therefore,  we have by~\eqref{conv unig rhon gn} the convergence
\begin{align*}
  \sup_{ t \in A  }\left|  \gn(t) -  g(t,\theta_0) \right| \convP 0.
\end{align*}
\end{proof}

 \begin{lemma}\label{lemma pcf bornitude}
  If \ref{assumption general domaine}-\ref{assumption general differentiabilty}, \ref{assumption pcf rm}-\ref{assumption pcf w} hold then $ j_g(.)$ is continuous on $[\rm,\rM]$.
 \end{lemma}
 
\begin{proof}
By~\eqref{expression pcf noyau DPP}, we have for all $t\in [\rm,\rM]$
\begin{align*}
|j_g(t)|=2  \left|w(t) \left(1-R_{\theta_0}(t)^2 \right)^{2c-2} R_{\theta_0}(t) R_{\theta_0}^{(1)}(t)\right| .
\end{align*}
By \ref{assumption general differentiabilty}, $R_{\theta_0}(.)$  and $R_{\theta_0}^{(1)}(.)$ are continuous on $[\rm,\rM]$. Further, by~\ref{assumption pcf rm}, $\rm>0$ and as noticed at the beginning of the proof of Theorem~\ref{theorem pcf normalite asymptotique}, for all $t>0$, $\left|R_{\theta_0}(t)\right|<1$. Thus by \ref{assumption general differentiabilty}, the function $t \mapsto \left(1-R_{\theta_0}(t)^2 \right)^{2c-2}$ is well defined and continuous on $[\rm,\rM]$.  Finally, by~\ref{assumption pcf w}, $w$ is continuous on $[\rm,\rM]$ so the lemma is proved. 
\end{proof}

To abbreviate, we define for all $n\in\N$ and $t\in [\rm,\rM]$, 
\begin{align*}
  H^g_n(t)&:=\wrho^{\, 2} \gn(t)- 2\rho_0 \wrho g(t,\theta_0). 
\end{align*}

\begin{lemma}\label{lemma variance limite pcf}
If \ref{assumption general domaine}-\ref{assumption general differentiabilty} and~\ref{assumption pcf rm}-\ref{assumption pcf differentiabilty}
  hold, we have for all $s\in\R^d$,
 \begin{align*}
\lim_{n\rightarrow +\infty} \Dn \mathrm{Var} \left( \int_\rm^\rM H^g_n(t) s^T j_g(t) dt \right)=  s^T \Sigma_{\rho_0,\theta_0} s
\end{align*}
with  $\Sigma_{\rho_0,\theta_0}$ defined as in Theorem~\ref{theorem pcf normalite asymptotique}. 
\end{lemma}

\begin{proof}
 Similarly to the proof of Lemma~\ref{lemma variance limite ripley}, we have by~\eqref{formula estimator pcf},
\begin{align*}
\int_\rm^\rM H^g_n(t)   s^T j_g(t) dt = \sum_{(x,y) \in \X^2} f_n (x,y) - \sum_{x\in\X} h_n(x) 
\end{align*}
where for all $n\in\N$,
\begin{align*}
 f_n(x,y):=   \1_{\lbrace x\in D_n \rbrace} \int_\rm^\rM \frac{k\left( \frac{t-|x-y|}{b_n} \right) \1_{\lbrace|x-y|>0, y\in  D_n \rbrace} }{ \sigma_d t^{d-1}b_n |D_n \cap D_n^{x-y}|}  s^T j_g(t) dt
 \end{align*}
 and 
 \begin{align*}
  h_n(x) = \frac{2 \rho_0 }{\Dn} \1_{\lbrace x\in D_n \rbrace} \int_\rm^\rM g(t,\theta_0)    s^T j_g(t) dt .
 \end{align*}
The result follows similarly as in the proof of Lemma~\ref{lemma variance limite ripley} using  Lemmas~7.1-7.3 in~\cite{bisciobrillingerTCL}.
 \end{proof}

\begin{lemma}\label{lemma majoration finale pour MCE g}
Assume that  \ref{assumption general domaine}-\ref{assumption general differentiabilty} and \ref{assumption pcf rm}-\ref{assumption pcf bn}  hold. For a given $s\in\R^d$ and all $n\in \N$, let  $f_{D_n}$  be defined for any $(x,y)\in \R^{2d}$ by
 \begin{multline*}
 f_{D_n}(x,y)\\ :=   \1_{\lbrace x\in D_n \rbrace }\int_\rm^\rM \left(\frac{k\left( \frac{t-|x-y|}{b_n} \right) \1_{\lbrace|x-y|>0, y\in  D_n \rbrace} }{ \sigma_d t^{d-1}b_n |D_n \cap D_n^{x-y}|} - \frac{2\rho_0  g(t,\theta_0)}{\Dn} \1_{\lbrace x-y=0\rbrace}    \right) s^T j_g(t) dt.
\end{multline*}
Then, there exists $M>0$ such that for all $(x,y)\in \R^{2d}$,
\begin{align*} 
 \left| f_{D_n}(x,y) \right| \leq \frac{|s| M\1_{\left\lbrace x\in D_n\right\rbrace}}{|D_n^{\circleddash \rM+T}|}  \left(  \frac{1}{\sigma_d r^{d-1}_{min}} \1_{\left\lbrace 0< |x-y| \leq \rM+ T\right\rbrace} + 2\rho_0 ||g||_{\infty} \1_{\lbrace x-y=0 \rbrace}  \right).
\end{align*}
\end{lemma}

\begin{proof}
By \ref{assumption pcf k}, for any $t\in [\rm,\rM]$ and $(x,y) \in \R^{2d}$,
\begin{align*}
 \left|k\left(\frac{t-|x-y|}{b_n}\right) \right| \1_{\left\lbrace |y-x|>0,\,  y\in D_n\right\rbrace} &\leq \left|k\left(\frac{t-|x-y|}{b_n}\right) \right| \1_{\left\lbrace 0<|y-x|<t + T b_n\right\rbrace} \\
                                                                                     &\leq \left| k\left(\frac{t-|x-y|}{b_n}\right)\right| \1_{\left\lbrace 0<|y-x|<t+T\right\rbrace} 
\end{align*}
whenever $b_n<1$ which, by \ref{assumption pcf bn}, we assume in the following without loss of generality. Thus, for any $t\in [\rm,\rM]$ and $(x,y) \in \R^{2d}$,
\begin{align}\label{majoration support pcf}
 \left| k\left(\frac{t-|x-y|}{b_n}\right)\right| \1_{\left\lbrace |y-x|>0,\,  y\in D_n\right\rbrace}  \leq \left|k\left(\frac{t-|x-y|}{b_n}\right)\right| \1_{\left\lbrace 0<|y-x|<\rM+T\right\rbrace}.
\end{align}
Further, by Lemma~\ref{lemma pcf bornitude}, $j_g$ is bounded on $[\rm,\rM]$ by a constant $M$ so by~\eqref{majoration support pcf} and Lemma~6.3 in~\cite{bisciobrillingerTCL}, we have
\begin{multline*}
\left| 1_{\lbrace x\in D_n \rbrace }\int_\rm^\rM \frac{k\left( \frac{t-|x-y|}{b_n} \right) \1_{\lbrace|x-y|>0, y\in  D_n \rbrace} }{ \sigma_d t^{d-1}b_n |D_n \cap D_n^{x-y}|}    s^T j_g(t) dt \right|\\ 
\leq   \1_{\left\lbrace x\in D_n\right\rbrace} \frac{|s| M }{|D_n^{\circleddash \rM+T}|} \frac{ \1_{\left\lbrace 0< |x-y| \leq \rM+ T\right\rbrace }}{\sigma_d r^{d-1}_{min}  b_n}   \int_{\rm}^{\rM}   \left| k\left(\frac{t-|x-y|}{b_n}\right) \right| dt .
\end{multline*}
Finally, the result follows by the last inequality, \ref{assumption general Krho} and  \ref{assumption pcf k}.
\end{proof}

\begin{lemma}\label{lemma pcf TCL}
 If  \ref{assumption general domaine}-\ref{assumption general differentiabilty} and~\ref{assumption pcf rm}-\ref{assumption pcf differentiabilty} hold, then
   \begin{align*}
\sqrt{\Dn}\int_\rm^\rM  \left[ \gn(t) - g(t,\theta_0) \right] j_g(t)dt \convl \mathcal{N} (0,\Sigma_{\rho_0,\theta_0})                                                                                           \end{align*}
with $\Sigma_{\rho_0,\theta_0}$ defined as in Theorem~\ref{theorem pcf normalite asymptotique}.
\end{lemma}

\begin{proof}
The arguments of this proof are similar the the ones of the proof of Lemma~\ref{lemma ripley TCL}. Notice that
\begin{multline}\label{egalite tcl pcf 1}
\rho_0^2 \sqrt{\Dn} \int_{\rm}^{\rM} \left[ \gn(t)- g(t,\theta_0) \right] j_g(t) dt =  \sqrt{\Dn} \left( \left[ \rho_0^2-\wrho^{\, 2} \right] \int_{\rm}^{\rM}  \gn(t) j_g(t) dt \right. \\\left. +   \int_{\rm}^{\rM} \left[ \wrho^{\, 2}\gn(t)- \E\left[\wrho^{\, 2}\gn(t)\right] \right] j_g(t) dt + \int_{\rm}^{\rM} \left[ \E\left[\wrho^{\, 2}\gn(t)\right]-  \rho_0^2 g(t,\theta_0) \right] j_g(t) dt \right)
\end{multline}
and
 \begin{multline}\label{terme negligeable preuve tcl pcf}
  \sqrt{\Dn} \left[ \rho_0-\wrho \right] \int_{\rm}^{\rM}  \gn(t) j_g(t)  dt =  \\ \sqrt{\Dn} \left[ \rho_0-\wrho \right] \int_{\rm}^{\rM}  \left[\gn(t) - g(t,\theta_0) \right] j_g(t) dt  +    \sqrt{\Dn}\left[ \rho_0-\wrho \right]\int_{\rm}^{\rM}  g(t,\theta_0) j_g(t) dt.
 \end{multline}
Denote 
\begin{align*}
 T_n &=  2\rho_0 \sqrt{\Dn}\left[ \rho_0-\wrho \right] \int_{\rm}^{\rM}  \gn(t) j_g(t) dt\\
 U_n &=  2\rho_0 \sqrt{\Dn}\left[ \rho_0-\wrho \right]  \int_{\rm}^{\rM} \left[\gn(t) - g(t,\theta_0) \right] j_g(t) dt\\
 V_n &=  \sqrt{\Dn} \int_{\rm}^{\rM} \left[ \E\left[\wrho^{\, 2}\gn(t)\right]-  \rho_0^2 g(t,\theta_0) \right] j_g(t) dt \\
 W_n &=  \sqrt{\Dn} \int_{\rm}^{\rM} \left[ \wrho^{\, 2}\gn(t)- 2\rho_0 \wrho g(t,\theta_0) - \left(\E\left[ \wrho^{\, 2}\gn(t)\right] - 2\rho_0^2 g(t,\theta_0) \right) \right] j_g(t) dt .
\end{align*}
Using~\eqref{taylor sur intensite ripley} in the proof of Lemma~\ref{lemma ripley TCL}, \eqref{egalite tcl pcf 1} and \eqref{terme negligeable preuve tcl pcf}, we get
\begin{align}\label{egalite tcl pcf short}
  \rho_0^2 &\sqrt{\Dn} \int_{\rm}^{\rM} \left[ \gn(t)-g(t,\theta_0) \right] j_g(t) dt = U_n + V_n + W_n +o\left( T_n \right).
\end{align}
We prove that $U_n+V_n + o(T_n)$ tends in probability to $0$ and we conclude by proving that $W_n$ tends in distribution to a Gaussian variable.
From Corollary~\ref{corollary convergence normal intensite DPP}, Lemmas~\ref{lemma conv univ pcf}-\ref{lemma pcf bornitude}  and  Slutsky's theorem, we have $U_n \convP 0$. Further, since $g$ is continuous on $[\rm,\rM]$ so bounded, we have by Lemma~\ref{lemma conv univ pcf} that $\gn$ is uniformly bounded in probability on $[\rm,\rM]$, see \cite[Prohorov's theorem]{vandervaart}. Thus, by Corollary~\ref{corollary convergence normal intensite DPP} and Lemma~\ref{lemma pcf bornitude}, $o(T_n) \convP 0$. Further, under \ref{assumption general domaine}-\ref{assumption general Krho} and \ref{assumption pcf rm}-\ref{assumption pcf differentiabilty}, we deduce from  Lemma 6.2 in~\cite{bisciobrillingerTCL}  that 
$\sup_{t\in[\rm,\rM]} \left(\E\left[\wrho^{\, 2}\gn(t)\right]-  \rho_0^2 g(t,\theta_0) \right)<\kappa b_n^2$ with $\kappa>0$, which combined with \ref{assumption pcf bn}  and  Lemma~\ref{lemma pcf bornitude} proves that $V_n \convP 0$. \medskip

We prove the  convergence in distribution of $W_n$ by the Cramer-Wold device. To shorten, denote for all $n\in \N$ and $s\in \R^p$,
\begin{align*}
 X^s_n := \int_{\rm}^{\rM} H_n^g(t) s^T j_g(t) dt.
\end{align*}
By Lemma~\ref{lemma pcf bornitude}, $j_g$ is  bounded on $[\rm,\rM]$ by a constant $M$. Then, since for all $t\in[\rm,\rM]$ and $n\in \N$,
\begin{align*}
 H_n^g(t) = \wrho^{\, 2} \gn(t)- \rho_0^2 g(t,\theta_0)  + (\rho_0 - \wrho) \rho_0 g(t,\theta_0) - \rho_0 \wrho g(t,\theta_0),
\end{align*}
we have
\begin{multline}\label{inegalite pour TCD Hng}
\E \left( \int_{\rm}^{\rM} \left| H_n^g(t) s^T j_g(t) \right| dt \right) \leq  |s| M \E \int_\rm^\rM    \left(\left| \wrho^{\, 2} \gn(t)- \rho_0^2 g(t,\theta_0) \right|\right) dt \\ +|s| M \left[ \E(| \rho_0 - \wrho |) + \E(\wrho) \right] \int_\rm^\rM |\rho_0 g(t,\theta_0)|  dt .
\end{multline}
By~\ref{assumption general differentiabilty}, $g(.,\theta_0)$ is bounded on $[\rm,\rM]$. Denote $||g||_\infty$ its maximum so by Cauchy Schwartz inequality, Jensen inequality and~\eqref{inegalite pour TCD Hng}, we have
\begin{multline}\label{inegalite pour TCD Hng 2}
 \int_{\rm}^{\rM}  \E\left| H_n^g(t) s^T j_g(t) \right| dt  \leq  |s| M (\rM-\rm)^{\frac{1}{2}} \Big( \E \int_\rm^\rM   \left( \wrho^{\, 2} \gn(t)- \rho_0^2 g(t,\theta_0)  \right)^2 dt \Big)^{\frac{1}{2}} \\ + |s|M(\rM-\rm)\rho_0 ||g||_\infty\left(  \E(| \rho_0 - \wrho |) + \E(\wrho) \right).
\end{multline}
By the same arguments as in the proof of Lemma~\ref{lemma conv univ pcf}, we have 
\begin{multline*}
 \E\left[ \int_\rm^\rM \left(\wrho^{\, 2} \gn(t) - \rho_0^2 g(t,\theta_0) \right)^2 dt \right] \\ = \frac{2 \rho_0^2}{b_n \Dn} \int_\rm^\rM \frac{g(t,\theta_0)}{ \sigma_d t^{d-1} } dt \int_\R k(x)^2 dx + O\left(\frac{1}{\Dn}\right) + O(b_n^4).
\end{multline*}
Thus by~\ref{assumption pcf bn}, $\E \left(\int_\rm^\rM \left(\wrho^{\, 2} \gn(t) - \rho_0^2 g(t,\theta_0) \right)^2 dt\right)$ tends to $0$. Moreover, as noticed in~\cite{heinrich1992minimum}, $\wrho$ converge in $L^1$ to $\rho_0$ so  $\E(| \rho_0 - \wrho |) + \E(\wrho)$ converges to $\rho_0$.  Hence, by~\eqref{inegalite pour TCD Hng}-\eqref{inegalite pour TCD Hng 2}, $ \E \int_{\rm}^{\rM} \left| H_n^g(t) s^T j_g(t) \right| dt$ is bounded. Then, by Fubini theorem,
\begin{align*}
 \E( X_n^s ) = \int_{\rm}^{\rM} \E \left( H_n^g(t) \right) s^T j_g(t) dt
\end{align*}
which implies that
\begin{align*}
 s^T W_n = \sqrt{\Dn} \left(X_n^s - \E( X_n^s) \right).
\end{align*}
By~\eqref{formula estimator pcf}, we have
\begin{align}\label{autre exression fonctionnelle pcf}
X^s_n = \sum_{(x,y) \in \X^2} f_{D_n} (x,y),
\end{align}
where $f_{D_n}(x,y)$ is given in Lemma~\ref{lemma majoration finale pour MCE g} and satisfies
\begin{align} \label{majoration fDn pcf lemma final}
 \left| f_{D_n}(x,y) \right| \leq \frac{|s| M\1_{\left\lbrace x\in D_n\right\rbrace}}{|D_n^{\circleddash \rM+T}|}  \left(  \frac{1}{\sigma_d r^{d-1}_{min}} \1_{\left\lbrace 0< |x-y| \leq \rM+ T\right\rbrace} + 2\rho_0 ||g||_{\infty} \1_{\lbrace x-y=0 \rbrace}  \right).
\end{align}
The right-hand term in~\eqref{majoration fDn pcf lemma final} is bounded and compactly supported. 
Therefore, by Lemma~\ref{lemma variance limite pcf} and Theorem~\ref{TCL jolivet modifie}, we have for all $s\in\R^p$ 
\begin{align*}
\sqrt{\Dn} \left(X_n^s - \E( X_n^s) \right) \convl  N(0, s^T\Sigma_{\rho_0,\theta_0} s ),
\end{align*}
which implies that $W_n \convl N(0, \Sigma_{\rho_0,\theta_0})$.
\end{proof}

  \section{Appendix}

\subsection{A general result for minimum contrast estimation}\label{section general result minimum}

We present in this section two general theorems concerning the consistency and asymptotic normality of the estimator defined  in~\eqref{definition thetan}. 
 Contrary to the results in Sections~\ref{section theorem ripley}-\ref{section theoreme pcf}, these theorems hold for an arbitrary stationary point process and an arbitrary statistic $J$, generalizing a study by~\cite{GuanSherman:07}. The results of Sections~\ref{section theorem ripley}-\ref{section theoreme pcf} are in fact consequences in the particular case of a DPP and $J=K$ or $J=g$, which simplifies the general assumptions below.

Let $\X$ be a stationary point process belonging to a parametric family indexed by, among possibly other parameters, $\theta\in \Theta$ where $\Theta\subset \R^p$, for a given $p\geq 1$. For any $t\in [\rm,\rM]$,  let $J(t,\theta)$ be  any real valued summary statistic of $\X$ that depends on $\theta$ (specific assumptions on $J$ are listed below). For any $t\in[\rm,\rM]$, let $\Jn(t)$ be an estimator of $J(t,\theta_0)$ where $\theta_0$ is the true parameter ruling the distribution of $\X$.
We denote by $J^{(1)}(t,\theta)$ and $J^{(2)}(t,\theta)$ the gradient, respectively the Hessian matrix, of  $J(t,\theta)$ with respect to $\theta$. Define for all $\theta \in \Theta$, 
\begin{align}\label{definition B}
 B(\theta):= \int_\rm^\rM w(t) J(t,\theta)^{2c-2} J^{(1)}(t,\theta) J^{(1)}(t,\theta)^T dt,
\end{align}
and for all $t\in[\rm,\rM]$,
\begin{align*}
 j(t)=w(t) J(t,\theta_0)^{2c-2} J^{(1)}(t,\theta_0).
\end{align*}
We consider the following assumptions.
\begin{enumerate}[label=($\mathcal{A}$\arabic*)]
 \item   $\Theta$ is a compact set with non-empty interior, $0\leq \rm < \rM$, $c\neq 0$ and $\left\lbrace D_n \right\rbrace_{ n\in \N }$ is a regular sequence of subsets of $\R^d$ in the sense of Definition~\ref{definition regular set}. \label{assumption Theta compact et Dn regular} 
 
 \item  $w$ is a positive and integrable function in $[\rm,\rM]$. \label{assumption w}
  
 \item  $J(.,.)$ and $J(.,. )^c$ are well defined continuous functions on $[\rm,\rM] \times \Theta$. Moreover, there exists a set $A\in [\rm,\rM]$ such that $[ \rm, \rM]\setminus A$ is of Lebesgue measure null and for all $t\in A$,  $\theta\in \Theta$, we have $J(t,\theta)>0$.
 \label{assumption continuite de J et param c}

 \item   
 There exists $n_0\in\N$ such that for all $n\geq n_0$, $\Jn(.)$ and $\Jn(.)^c$  are almost surely bounded on $[\rm,\rM]$.
 \label{assumption Jn borne positive}
 \item  There exists a set $A\in [\rm,\rM]$ such that $[ \rm, \rM]\setminus A$ is of Lebesgue measure null and 
 \begin{align*}
 \sup_{ t \in A} \left| \Jn(t) - J(t,\theta_0) \right| \convP 0.
 \end{align*}\label{assumption convergence uniforme}\vspace*{-0.6cm}
   
   \item  For $\theta_1 \neq \theta_2$, there exists a set $A$ of positive Lebesgue measure such that
\begin{align*}
J(t,\theta_1) \neq J(t,\theta_2) , \quad \forall t\in A.
\end{align*} \label{assumption identifiability}\vspace*{-0.8cm}
  
   \item For all $t \in [ \rm, \rM]$, $J^{(1)}(t,\theta)$ and $J^{(2)}(t,\theta)$ exist, are continuous with respect to $\theta$ and uniformly bounded with respect to $t\in[\rm,\rM]$ and $\theta\in\Theta$. \label{assumption J derivable 2}
 
 \item There exists $M>0$ such that for all $(t,\theta) \in [\rm,\rM] \times \Theta$ and $a\in\lbrace c-2, 2c-2\rbrace$, $\big| J(t,\theta) \big| ^a  \leq M $. 
\label{assumption integrability pour B}
 
 \item The matrix $B(\theta_0)$ is invertible. \label{assumption B inversible}

\item[$(\mathcal{TCL})$] There exists $m\in \R$ and a covariance matrix $\Sigma$ such that
  \begin{align*}
\sqrt{\Dn}\int_\rm^\rM  \left[ \Jn(t) - J(t,\theta_0) \right] j(t)dt \convl \mathcal{N} (m,\Sigma).                                                                                   \end{align*}
\end{enumerate} 
Further, define $(\mathcal{A}5)'$ as the assumption \ref{assumption convergence uniforme} with the convergence in probability replaced by the almost sure convergence.

\begin{theorem}\label{theoreme consistency}
Let $\X$ be a stationary point process with distribution ruled by a given $\theta_0$, assumed to be an interior point of $\Theta$. For all $n\in \N$, let $U_n$ be defined as in~\eqref{expression Un discrepancy measure}. Assume that\ref{assumption Theta compact et Dn regular}-\ref{assumption identifiability} hold. Then, the minimum contrast estimator $\ttheta$ defined by
\begin{align}\label{definition thetatilde}
 \ttheta = \argmin_{\theta \in \Theta}  U_n(\theta)
\end{align}
exists almost surely, is consistent for $\theta_0$ and strongly consistent if $(\mathcal{A}5)'$ holds. 
\end{theorem}

\begin{proof}
For a sequence $\lbrace \theta_m\rbrace_{m\in \N}$ belonging to $\Theta$, we have for all $n\in\N$,
\begin{multline}\label{majoration pour continuite Un}
 \left| U_n(\theta_m) - U_n( \theta ) \right| \leq \\ \int_{\rm}^\rM |w(t)|\left( \big| 2\Jn(t)^c \big|  \left| J(t,\theta_m)^c-J(t,\theta)^c\right| + \left|J(t,\theta_m)^{2c} - J(t,\theta)^{2c}\right| \right) dt.
\end{multline}
Denote $A$ the intersection of the sets defined in  \ref{assumption continuite de J et param c} and \ref{assumption convergence uniforme}. By \ref{assumption continuite de J et param c},  $J(.,.)^c$ is continuous on $[\rm,\rM] \times \Theta$ which is compact by \ref{assumption Theta compact et Dn regular}. We deduce that
\begin{align*}
 \sup_{t \in [\rm,\rM]} \left| J(t,\theta_m)^c - J(t,\theta)^c \right| \leq K.
\end{align*}
By \ref{assumption continuite de J et param c}-\ref{assumption Jn borne positive}, for all $\theta\in \Theta$, $J(t,\theta)^c$ and $\Jn(t)^c$ are almost surely bounded on $[\rm,\rM]$, for all $n$ large enough. Further, by \ref{assumption w}, $w$ is integrable on $[\rm,\rM]$ thus, by \eqref{majoration pour continuite Un} and the dominated convergence theorem, we have the convergence
\begin{align*}
  \left| U_n(\theta_m) - U_n( \theta ) \right| \xrightarrow[\theta_m \rightarrow \theta]{a.s.}  0.
\end{align*}
Therefore, for all $n$ large enough, $U_n$ is almost surely continuous so the almost sure existence of $\ttheta$ follows by \ref{assumption Theta compact et Dn regular}.
Define for all $\theta \in \Theta$,
\begin{align}\label{definition U* en fonction de U}
 U_n^*(\theta)= U_n(\theta) - U_n(\theta_0).
\end{align}
By \eqref{expression Un discrepancy measure} and \eqref{definition U* en fonction de U},
\begin{multline*}
 U_n^*(\theta) =  2 \int_\rm^\rM w(t) \big[\Jn(t)^c-J(t,\theta_0)^c  \big]   \big[J(t,\theta_0)^c - J(t,\theta)^c \big] dt \\ + \int_\rm^\rM w(t) \big[J(t,\theta_0)^c - J(t,\theta)^c \big]^2 dt.
\end{multline*}
Note that from~\eqref{definition U* en fonction de U} $U_n^*( \ttheta ) \leq  U_n^*(\theta_0) = 0$, so
\begin{align}\label{majoration pour consistency}
 \int_\rm^\rM w(t) \big[J(t,&\theta_0)^c - J(t,\ttheta)^c \big]^2 dt \notag\\ 
 &\leq 2 \int_\rm^\rM w(t) \big|\Jn(t)^c-J(t,\theta_0)^c  \big|   \big|J(t,\theta_0)^c - J(t,\ttheta)^c \big| dt.
\end{align}
By \ref{assumption continuite de J et param c}-\ref{assumption Jn borne positive}, $J(.,.)^c$  is continuous on $[\rm,\rM] \times \Theta$ and for $n$ large enough, $\Jn(.)^c$ is almost surely bounded on $[\rm,\rM]$ so by~\ref{assumption convergence uniforme}, the right-hand term in~\eqref{majoration pour consistency} tends in probability to $0$. Hence, we have
\begin{align*}
  \int_\rm^\rM w(t) \big[J(t,&\theta_0)^c - J(t,\ttheta)^c \big]^2 dt \convP 0.
\end{align*}
It follows by \ref{assumption w} and \ref{assumption identifiability} that $\ttheta$ converges in probability to $\theta_0$. Finally, by a similar argument, we prove by~\eqref{majoration pour consistency} that this last convergence is almost sure if $(\mathcal{A}5)'$ holds.
\end{proof}

\begin{theorem}\label{theoreme normalite asymptotique}
Under the same setting as in Theorem~\ref{theoreme consistency}, if in addition \ref{assumption J derivable 2}-\ref{assumption B inversible} and $(\mathcal{TCL})$ hold true, then 
\begin{align*}
 \sqrt{\Dn}(\ttheta -\theta_0) \convl  \mathcal{N} \left( m,B(\theta_0)^{-1} \Sigma \left( B(\theta_0)^{-1}  \right)^T \right)
\end{align*}
where $B$ is defined as in~\eqref{definition B} and $\Sigma$ comes from $\mathcal{(TCL)}$.
\end{theorem}

\begin{proof}
Denote by $A$ the intersection of the sets defined in \ref{assumption continuite de J et param c} and \ref{assumption convergence uniforme}. Then, by \ref{assumption continuite de J et param c},~\ref{assumption J derivable 2} and~\ref{assumption integrability pour B},  we see that $U_n$ is almost surely twice differentiable on $\Theta$ and that we can differentiate twice under the integral sign. Thus, by the mean value theorem, 
for all $j= 1,\ldots,p$, there exists $s\in (0,1)$ and  $\theta^*_j = \theta_0 +s (\ttheta-\theta_0)$ such that
\begin{align*} 
\partial_j U_n(\ttheta) - \partial_j U_n(\theta_0) = \left( \partial_{ij}^2 U_n(\theta^*_j)  \right)_{i=1,\ldots,p}   \left(\ttheta-\theta_0 \right).
\end{align*}
To shorten, denote by $U_n^{(1)}$ the gradient of $U_n$ and by $U_n^{(2)}(\theta_n^*)$ the matrix with entries  $\partial_{ij}^2 U_n(\theta^*_j)$. 
Since $U_n$ is minimal at $\ttheta$, $ U_n^{(1)}(\ttheta)=0$ and the last equation becomes
\begin{align}\label{norma asympt 1}
U_n^{(2)}(\theta_n^*) ( \ttheta-\theta_0 ) &= - U_n^{(1)}(\theta_0) \nonumber\\
& =  2c \int_\rm^\rM w(t) \big[ \Jn(t)^c - J(t,\theta_0)^c \big] J(t,\theta_0)^{c-1} J^{(1)}(t,\theta_0) dt.
\end{align}
Note that by~\ref{assumption continuite de J et param c} and~\ref{assumption integrability pour B}, $J(.,\theta_0)^{c-1}$ is bounded on $A$ and strictly positive. Thus, by  \ref{assumption Jn borne positive}, we can use the Taylor expansion of the function $x \mapsto x^c$ so, for all $t\in A$, 
\begin{align*} 
 \Jn(t)^c - J(t,\theta_0)^c = c J(t,\theta_0)^{c-1} \left(\Jn(t)-J(t,\theta_0) \right) + o \left( \Jn(t)-J(t,\theta_0) \right) .
\end{align*}
 Therefore, by~\ref{assumption convergence uniforme},~\eqref{norma asympt 1} and the last equation,
\begin{align}\label{egalite norm asympt 1}
\sqrt{\Dn}\, U_n^{(2)}(\theta_n^*) ( \ttheta - \theta_0)  =2 c^2 A_n(\theta_0) + o( A_n(\theta_0) )
\end{align}
where
\begin{align*}
  A_n(\theta_0)= \sqrt{\Dn} \int_A  \left[ \Jn(t) - J(t,\theta_0) \right] j(t) dt.
\end{align*}
By $(\mathcal{TCL})$, we have $2c^2 A_n(\theta_0) \convl 2c^2 N(m,  \Sigma)$. Hence, by Slutsky's theorem and~\eqref{egalite norm asympt 1}, 
\begin{align}\label{convergence part Un asympt}
 \sqrt{\Dn}\, U_n^{(2)}(\theta_n^*) ( \ttheta - \theta_0) \convl 2c^2 N(m,  \Sigma).
\end{align}
Moreover, we have that
\begin{align}\label{egalite B}
\, U_n^{(2)}(\theta_n^*) = 2c^2 B(\theta_n^*)- E_n
\end{align}
where $B$ is as in~\eqref{definition B} and 
\begin{multline*}
 E_n:= 2c(c-1)  \int_\rm^\rM w(t) \big[ \Jn(t)^c - J(t,\theta_n^*)^c \big] J(t,\theta_n^*)^{c-2} J^{(1)}(t,\theta_n^*) J^{(1)}(t,\theta_n^*)^T dt   \\ 
      +2c \int_\rm^\rM w(t)  \big[ \Jn(t)^c - J(t,\theta_n^*)^c \big] J(t,\theta_n^*)^{c-1} J^{(2)}(t,\theta_n^*) dt.
\end{multline*}
By Theorem~\ref{theoreme consistency}, $\ttheta \convP \theta _0$ so $\theta_n^* \convP\theta _0$. Then, by \ref{assumption J derivable 2}-\ref{assumption integrability pour B}, $E_n$ tends in probability to $0$. Note that by continuity of $J(.,\theta)$ for all $\theta \in \Theta$, the integrability  on $[\rm,\rM]$ of $w(.) J(.,\theta)^{c-2}$ implies the one of $w(.) J(.,\theta)^{c-1}$. 
Further, we deduce by \ref{assumption continuite de J et param c}, \ref{assumption J derivable 2} and \ref{assumption integrability pour B} that $(t,\theta) \mapsto J(t,\theta)^{2c-2} J^{(1)}(t,\theta)J^{(1)}(t,\theta)^{T}$ is continuous with respect to $\theta\in \Theta$ and uniformly bounded for $t \in [\rm,\rM]$. Thus, by the dominated convergence theorem,
\begin{align*}
 B(\theta_n^*)  \convP  B(\theta_0).
\end{align*}
By~\ref{assumption B inversible}, $B(\theta_0)$ is invertible so by~\eqref{egalite B} 
\begin{align}\label{convergence et inversibilite part Un}
 U_n^{(2)}(\theta_n^*) B(\theta_0)^{-1} \convP 2c^2 .
\end{align}
Since the group of invertible matrix is an open set, it follows from the last convergence that for $n$ large enough, $U_n^{(2)}(\theta_n^*)$ is invertible so we can write
\begin{align*}
 B(\theta_0) \sqrt{\Dn} (\ttheta-\theta_0) =  B(\theta_0) \left[ U_n^{(2)}(\theta_n^*)\right]^{-1} U_n^{(2)}(\theta_n^*) \sqrt{\Dn} (\ttheta-\theta_0).
\end{align*}
By~\eqref{convergence part Un asympt}-\eqref{convergence et inversibilite part Un} and Slutsky's theorem, we get
\begin{align*}
 B(\theta_0) \sqrt{\Dn}  ( \ttheta - \theta_0) \convl  N(m,  \Sigma)
\end{align*}
and the conclusion of the theorem follows.
\end{proof}

  \subsection{Auxiliary results}  
  The two following lemmas are of topological nature and useful for the proofs of Theorems~\ref{theorem ripley normalite asymptotic}-\ref{theorem pcf normalite asymptotique}.
  
  \begin{lemma}\label{lemma point dehors cpct}
 For all $p\geq 1$, let $\Xi$ be a compact convex set in $\R^p$. Then, for all $y\in \R^p \setminus \Xi$ and $\delta \geq 0$, $  B(y,\delta) \nsubseteq \Xi^{\oplus \delta}$.
\end{lemma}
\begin{proof}
 Since $\Xi$ is a closed convex set,  the projection of $y$ onto $\Xi$, denoted by $p_{\Xi}(y)$, is the unique element belonging to $\Xi$ that, for all $u\in \Xi$, verifies
 \begin{align}\label{caracterisation projete convexe ferme}
  (y-p_{\Xi}(y) ) . ( u-p_{\Xi}(y)) \leq 0.
 \end{align}
For all $\delta\geq 0$,  the line $(y,p_{\Xi}(y))$ intersects $\partial B(y,\delta)$ at two points, one inside the segment $[y,p_{\Xi}(y)]$ and the other, that we denote by $v$, outside the segment. Thus, there exists $t>1$ such that $v = p_{\Xi}(y) + t(y- p_{\Xi}(y) )$. Notice that for all $u\in \Xi$,
 \begin{align*}
  (v-p_{\Xi}(y) ) . ( u-p_{\Xi}(y)) = t  (y-p_{\Xi}(y) ) . ( u-p_{\Xi}(y)).
 \end{align*}
Thus, as $t> 1$,  we deduce from~\eqref{caracterisation projete convexe ferme} and the last equation that $p_{\Xi}(y)$ is the projection of $v$ onto $\Xi$. It follows that $d(v,\Xi) = d(y,\Xi) + \delta$ and  as $y\notin \Xi$ and $\Xi$ is closed,   $d(v,\Xi) > \delta$. Therefore,  $v\notin \Xi^{\oplus \delta}$ but $v\in \partial B(y,\delta)$ by construction so $B(y,\delta) \nsubseteq \Xi^{\oplus \delta}$.
\end{proof}

\begin{lemma}\label{lemma boule incluse dans grossisement}
 For $p\geq1$, let $\Theta$ be a convex compact set in $\R^p$ and $\lbrace \Theta_n \rbrace_{n\in\N}$ be a sequence of convex compact sets in $\R^p$ that converges to $\Theta$ with respect to the Hausdorff distance. Let $r\geq 0$ and $x$ be an interior point of $\Theta$ such that $B(x,r)$ belongs to the interior of $\Theta$. Then, there exists $N\in \N$ such that for all $n\geq N$,
 \begin{align*}
  B(x,r) \subset \Theta_n.
 \end{align*}
\end{lemma}

\begin{proof}
Since $B(x,r)$ belongs to the interior of $\Theta$, there exists $\delta>0$ such that
\begin{align}\label{inclusion grossisement boule}
 B(x,r+\delta) \subset \Theta. 
 \end{align}
 Assume that the lemma is wrong, then for all $N\in \N$, there exists $n\geq N$ such that $B(x,r) \nsubseteq \Theta_n$. Denote $y$ a point in $B(x,r)$ that does not belong to $\Theta_n$. By Lemma~\ref{lemma point dehors cpct}, $B(y,\delta) \nsubseteq \Theta_n^{\oplus \delta}$. But by~\eqref{inclusion grossisement boule}, $B(y,\delta) \subset \Theta$ so $\Theta \nsubseteq \Theta_n^{\oplus \delta}$ which contradicts the convergence of the sequence  $\lbrace \Theta_n \rbrace_{n\in\N}$ to $\Theta$.  \end{proof}

The following  theorem appears in~\cite{jolivetTCL} in a slightly less general framework, see also~\cite{krickeberg1980}, and is proved in~\cite{bisciobrillingerTCL} in  its present form. It is used in the proofs of our main results, Theorems~\ref{theorem ripley normalite asymptotic} and \ref{theorem pcf normalite asymptotique}.

  \begin{theorem}\label{TCL jolivet modifie}
Let $\lbrace D_n\rbrace_{n\in \N}$ and $\lbrace \widetilde{D}_n \rbrace_{n\in\N}$ be two sequences of regular sets in the sense of Definition~\ref{definition regular set} such that $\frac{|\widetilde{D}_n|} {|D_n|}\xrightarrow{n \rightarrow + \infty} \kappa$ for a given $\kappa >0$. For all $n\in\N$,  let $\lbrace f_{D_n} \rbrace_{n\in\N} $ be a family of functions from $\R^{dp}$ into $\R$. Assume that there exists a bounded function $F$ from $\R^{d(p-1)}$ into $\R^+$ with compact support such that for all $n\in \N$ and $(x_1,\ldots,x_p) \in \R^{dp}$,
 \begin{align}\label{majoration ft cumulant}
  |f_{D_n}(x_1, \ldots, x_p)| \leq \frac{ 1}{|\widetilde{D}_n|} \1_{\lbrace x_1 \in D_n\rbrace} F(x_2-x_1,\ldots,x_p-x_1) .
 \end{align}
Assume further that the point process $\X$ is ergodic, admits moment of any order and is Brillinger mixing. Then, for all $k\geq 2$, we have
 \begin{align}\label{comportement asymptotic cumulant}
 \cum_k \left( \sqrt{\Dn}  N_p\left(f_{D_n} \right) \right) =  O \left( \Dn^{1- \frac{k}{2}} \right).
 \end{align}
 Moreover, if there exists  $\sigma>0$ such that 
 \begin{align}\label{condition limite variance tcl jolivet}
  \mathrm{Var}  \left(\sqrt{\Dn} N_p\left(f_{D_n}\right)\right) \xrightarrow[n\rightarrow +\infty]{} \sigma^2,
 \end{align}
then we have the convergence
  \begin{align}\label{convergence tcl jolivet}
\sqrt{\Dn}\left[  N_p\left(f_{D_n}\right) - \E\left( N_p \left(f_{D_n} \right) \right) \right] \convl \mathcal{N}(0,\sigma^2)
 \end{align}
 and the convergence of all moments to the corresponding moments of $\mathcal{N}(0,\sigma^2)$.
\end{theorem}
  
  As a corollary when $p=1$, we retrieve a theorem from \cite{SoshnikovGaussianLimit} giving the asymptotic normality of the estimator of the intensity of a DPP. 
\begin{corollary}\label{corollary convergence normal intensite DPP}
 Let $\X$ be a DPP with kernel $C$ verifying the condition $\K(\rho)$ for a given $\rho>0$  and  $\lbrace D_n\rbrace_{n\in \N}$ be a family of regular sets. Define for all $n\in\N$,
\begin{align*}
 \widehat{\rho}_n =\frac{1}{|D_n|} \sum_{x\in \X} \1_{\lbrace x \in D_n\rbrace}.
\end{align*}
We have the convergence
\begin{align*}
\sqrt{\Dn} \left(  \wrho - \rho \right) \convl N(0,\sigma^2)
\end{align*}
where $\sigma^2 =  \lim_{n\rightarrow +\infty} Var  \left(\sqrt{\Dn} \wrho \right) =\rho -\int_{\R^d} C(x)^2 dx$.
\end{corollary}

\bibliographystyle{acm}
 \bibliography{Main_Biblio.bib}

\begin{thebibliography}{10}

\bibitem{BaddelyRubakTurner15}
{\sc Baddeley, A., Rubak, E., and Turner, R.}
\newblock {\em Spatial Point Patterns: Methodology and Applications with {R}}.
\newblock Chapman and Hall/CRC Press, London, 2015.
\newblock In press.

\bibitem{baddeley:turner:05}
{\sc Baddeley, A., and Turner, R.}
\newblock {spatstat}: An {R} package for analyzing spatial point patterns.
\newblock {\em Journal of Statistical Software 12}, 6 (2005), 1--42.

\bibitem{billingsleyprobaetmesure1979}
{\sc Billingsley, P.}
\newblock {\em Probability and measure}.
\newblock John Wiley \& Sons, New York-Chichester-Brisbane, 1979.
\newblock Wiley Series in Probability and Mathematical Statistics.

\bibitem{bisciobrillingerTCL}
{\sc Biscio, C. A.~N., and Lavancier, F.}
\newblock Brillinger mixing of determinantal points processes and statistical
  applications.
\newblock {\em arXiv:1507.06506\/} (2015).

\bibitem{bisciobernoulli}
{\sc Biscio, C. A.~N., and Lavancier, F.}
\newblock Quantifying repulsiveness of determinantal point processes.
\newblock {\em To appear in Bernoulli\/} (2015).

\bibitem{daleyvol1}
{\sc Daley, D.~J., and Vere-Jones, D.}
\newblock {\em An Introduction to the Theory of Point Processes, Vol. I,
  Elementary Theory and Methods}, 2~ed.
\newblock Springer, 2003.

\bibitem{daleyvol2}
{\sc Daley, D.~J., and Vere-Jones, D.}
\newblock {\em An Introduction to the Theory of Point Processes, Vol. II,
  General Theory and Structure}, 2~ed.
\newblock Springer, 2008.

\bibitem{deng2014ginibre}
{\sc Deng, N., Zhou, W., and Haenggi, M.}
\newblock The {G}inibre point process as a model for wireless networks with
  repulsion.
\newblock {\em Wireless Communications, IEEE Transactions on 14}, 1 (2015),
  107--121.

\bibitem{Diggle:03}
{\sc Diggle, P.}
\newblock {\em Statistical Analysis of Spatial Point Patterns}, second~ed.
\newblock Hodder Arnold, London, 2003.

\bibitem{GuanSherman:07}
{\sc Guan, Y., and Sherman, M.}
\newblock On least squares fitting for stationary spatial point processes.
\newblock {\em Journal of the Royal Statistical Society. Series B. Statistical
  Methodology 69}, 1 (2007), 31--49.

\bibitem{heinrich1992minimum}
{\sc Heinrich, L.}
\newblock Minimum contrast estimates for parameters of spatial ergodic point
  processes.
\newblock In {\em Transactions of the 11th Prague conference on random
  processes, information theory and statistical decision functions\/} (1992),
  vol.~479, p.~492.

\bibitem{heinrichasymptotic:13}
{\sc Heinrich, L.}
\newblock Asymptotic methods in statistics of random point processes.
\newblock In {\em Stochastic Geometry, Spatial Statistics and Random Fields},
  vol.~2068 of {\em Lecture Notes in Mathematics}. Springer, Heidelberg, 2013,
  pp.~115--150.

\bibitem{hough2009zeros}
{\sc Hough, J.~B., Krishnapur, M., Peres, Y., and Vir{\'a}g, B.}
\newblock {\em Zeros of {G}aussian Analytic Functions and Determinantal Point
  Processes}, vol.~51 of {\em University Lecture Series}.
\newblock American Mathematical Society, Providence, RI, 2009.

\bibitem{jolivetTCL}
{\sc Jolivet, E.}
\newblock Central limit theorem and convergence of empirical processes for
  stationary point processes.
\newblock In {\em Point processes and queuing problems (Colloqium, {K}eszthely,
  1978)}, vol.~24 of {\em Colloq. Math. Soc. J\'anos Bolyai}. North-Holland,
  Amsterdam-New York, 1981, pp.~117--161.

\bibitem{krickeberg1980}
{\sc Krickeberg, K.}
\newblock Processus ponctuels en statistique.
\newblock In {\em Tenth {S}aint {F}lour {P}robability {S}ummer {S}chool---1980
  ({S}aint {F}lour, 1980)}, vol.~929 of {\em Lecture Notes in Math.} Springer,
  Berlin-New York, 1982, pp.~205--313.

\bibitem{Kulesza:Taskar:12}
{\sc Kulesza, A., and Taskar, B.}
\newblock Determinantal point processes for machine learning.
\newblock {\em Foundations and Trends in Machine Learning 5\/} (2012),
  123--286.

\bibitem{lavancier_moller15}
{\sc Lavancier, F., and M{\o}ller, J.}
\newblock Modelling aggregation on the large scale and regularity on the small
  scale in spatial point pattern datasets.
\newblock {\em Scandinavian Journal of Statistics\/} (2015).
\newblock To appear.

\bibitem{lavancier_extended}
{\sc Lavancier, F., M{\o}ller, J., and Rubak, E.}
\newblock Determinantal point process models and statistical inference :
  Extended version.
\newblock {\em arXiv:1205.4818v5\/} (2014).

\bibitem{lavancierpublish}
{\sc Lavancier, F., M{\o}ller, J., and Rubak, E.}
\newblock Determinantal point process models and statistical inference.
\newblock {\em Journal of the Royal Statistical Society: Series B (Statistical
  Methodology) 77}, 4 (2015), 853--877.

\bibitem{Lavancier:Rochet:2014}
{\sc Lavancier, F., and Rochet, P.}
\newblock A general procedure to combine estimators.
\newblock {\em Computational Statistics and Data Analysis\/} (2015).
\newblock To appear.

\bibitem{macchi1975coincidence}
{\sc Macchi, O.}
\newblock The coincidence approach to stochastic point processes.
\newblock {\em Advances in Applied Probability 7\/} (1975), 83--122.

\bibitem{Miyoshi:Shirai13}
{\sc Miyoshi, N., and Shirai, T.}
\newblock A cellular network model with ginibre configurated base stations.
\newblock Technical report, Department of Mathematical and Computing Sciences
  Tokyo Institute of Technology, series B: Applied Mathematical Science, 2013.

\bibitem{mollerstatisticalinference}
{\sc M{\o}ller, J., and Waagepetersen, R.~P.}
\newblock {\em Statistical Inference and Simulation for Spatial Point
  Processes}, vol.~100 of {\em Monographs on Statistics and Applied
  Probability}.
\newblock Chapman \& Hall/CRC, Boca Raton, FL, 2004.

\bibitem{zessinergodic}
{\sc Nguyen, X.-X., and Zessin, H.}
\newblock Ergodic theorems for spatial processes.
\newblock {\em Z. Wahrsch. Verw. Gebiete 48}, 2 (1979), 133--158.

\bibitem{R:15}
{\sc {R Core Team}}.
\newblock {\em R: A Language and Environment for Statistical Computing}.
\newblock R Foundation for Statistical Computing, Vienna, Austria, 2015.

\bibitem{sasvari2013multivariate}
{\sc Sasv{\'a}ri, Z.}
\newblock {\em Multivariate Characteristic and Correlation Functions}, vol.~50.
\newblock Walter de Gruyter, 2013.

\bibitem{ShiraiTakahashi1:03}
{\sc Shirai, T., and Takahashi, Y.}
\newblock Random point fields associated with certain {F}redholm determinants.
  {I}. {F}ermion, {P}oisson and boson point processes.
\newblock {\em Journal of Functional Analysis 205}, 2 (2003), 414--463.

\bibitem{Soshnikov:00}
{\sc Soshnikov, A.}
\newblock Determinantal random point fields.
\newblock {\em Russian Mathematical Surveys 55\/} (2000), 923--975.

\bibitem{SoshnikovGaussianLimit}
{\sc Soshnikov, A.}
\newblock Gaussian limit for determinantal random point fields.
\newblock {\em The Annals of Probability 30}, 1 (2002), 171--187.

\bibitem{stein1971fourier}
{\sc Stein, E., and Weiss, G.}
\newblock {\em Introduction to Fourier Analysis on Euclidean Spaces (PMS-32)},
  vol.~1.
\newblock Princeton university press, 1971.

\bibitem{vandervaart}
{\sc van~der Vaart, A.~W.}
\newblock {\em Asymptotic Statistics}, vol.~3 of {\em Cambridge Series in
  Statistical and Probabilistic Mathematics}.
\newblock Cambridge University Press, 1998.

\bibitem{Waagepetersentwostepestimation}
{\sc Waagepetersen, R., and Guan, Y.}
\newblock Two-step estimation for inhomogeneous spatial point processes.
\newblock {\em Journal of the Royal Statistical Society. Series B. Statistical
  Methodology 71}, 3 (2009), 685--702.

\end{thebibliography}

\end{document}